\documentclass{amsart}
\usepackage[T1]{fontenc}
\usepackage{amsmath,a4wide}
\usepackage{amssymb}
\usepackage{amscd}
\usepackage{epsfig}
\usepackage[T1]{fontenc}
\usepackage{epsfig}
\usepackage{amsmath}
\usepackage{graphics}
\usepackage{graphicx}
\usepackage{subfigure}
\usepackage{mathrsfs}
\usepackage{epstopdf}
\usepackage{color}
\usepackage{booktabs}
\usepackage{ifpdf}
\usepackage{cite}
\usepackage{amsthm}
\usepackage{amsmath}
\definecolor{myblue}{rgb}{0.2,0,0.9}
\definecolor{blue-violet}{rgb}{0.54, 0.17, 0.89}
\usepackage{enumitem}

\usepackage{xspace}

\usepackage{tikz}
\usetikzlibrary{arrows,intersections}
\usepackage{boldline}
\usepackage{multirow}
\usepackage[margin=1.4in]{geometry}

\usepackage{matlab-prettifier}

\usepackage{varioref}
\usepackage{xr-hyper}
\usepackage{hyperref}
\hypersetup{
	colorlinks,linkcolor=blue,citecolor=blue,filecolor=red,urlcolor=blue}

\definecolor{myblue}{rgb}{0.2,0,0.9}
\definecolor{blue_violet}{rgb}{0.54, 0.17, 0.89}
\definecolor{darkgreen}{rgb}{0,0.35,0}

\usepackage{algorithm,algorithmic}
\usepackage{caption}

\newlength\myindent
\makeatletter
\DeclareRobustCommand*\cal{\@fontswitch\relax\mathcal}
\makeatother

\newtheorem{thm}{Theorem}[section]
\newtheorem{pro}[thm]{Proposition}

\newtheorem{lem}[thm]{Lemma}
\numberwithin{equation}{section}

\theoremstyle{definition}
\newtheorem{rem}[thm]{Remark}
\newtheorem{dfn}[thm]{Definition}
\newtheorem{as}[thm]{Assumption}

\usepackage{xparse}
\makeatletter
\RenewDocumentCommand{\title}{om}{%
	\IfNoValueTF{#1}
	{\gdef\shorttitle{}}
	{\gdef\shorttitle{#1}}%
	\gdef\@title{#2}%
}
\makeatother

\title{Numerical method for nonlinear Kolmogorov PDEs\\ via sensitivity analysis}
\author{Daniel Bartl}
\address{Department of Mathematics, Department of Statistics and Data Science, National University of Singapore}
\email{bartld@nus.edu.sg}

\author{Ariel Neufeld}
\address{Division of Mathematical Sciences, Nanyang Technological University}
\email{ariel.neufeld@ntu.edu.sg}

\author{Kyunghyun Park}
\address{Division of Mathematical Sciences, Nanyang Technological University}
\email{kyunghyun.park@ntu.edu.sg}

\date{\today.}

\thanks{2020 \textit{MSC classifications.} 90C31, 60G65, 60H05, 91G10}
\thanks{\textit{Funding:}
D.\;Bartl is grateful for financial support through the Austrian Science Fund (grant doi: 10.55776/P34743 and 10.55776/ESP31), the Austrian National Bank (Jubil\"aumsfond, project 18983), and a  Presidential Young Professorship grant (`Robust statistical learning for complex data').
A.\;Neufeld gratefully acknowledges financial support by the Nanyang Assistant Professorship Grant (NAP Grant) {\it Machine Learning based Algorithms in Finance and Insurance} and the MOE AcRF Tier 2 Grant {\it MOE-T2EP20222-0013}. K.~Park acknowledges financial support by the Presidential Postdoctoral Fellowship of Nanyang Technological University and the National Research Foundation of Korea (grant DOI: RS-2025-02633175).}

\begin{document}
\begin{abstract}
		We examine nonlinear Kolmogorov partial differential equations (PDEs). Here the nonlinear part of the PDE comes from its Hamiltonian where one maximizes over all possible drift and diffusion coefficients which fall within a $\varepsilon$-neighborhood of pre-specified baseline coefficients. Our goal is to quantify and compute how sensitive those PDEs are to such a small nonlinearity, and then use the results to develop an efficient numerical method for their approximation. We show that as $\varepsilon\downarrow 0$, the nonlinear Kolmogorov PDE equals the linear Kolmogorov PDE defined with respect to the corresponding baseline coefficients plus $\varepsilon$ times a correction term which can be also characterized by the solution of another linear Kolmogorov PDE involving the baseline coefficients. As these linear Kolmogorov PDEs can be efficiently  solved in high-dimensions by exploiting their Feynman-Kac representation, our derived sensitivity analysis then provides a Monte Carlo based numerical method which can efficiently solve these nonlinear Kolmogorov equations. We establish an error and complexity analysis for our numerical method. Moreover, we present numerical results with up to 100 dimensions to demonstrate our theory and support the applicability of our numerical\;method.
\end{abstract}

\vspace*{-0.2cm}
\maketitle

\section{Introduction}\label{sec:intro}
Kolmogorov partial differential equations (PDEs) are widely used to describe the evolution of  underlying diffusion processes over time. These PDEs are applied in various fields, for instance
to model dynamics in physics and chemistry (e.g., \cite{van1992stochastic,pascucci2005kolmogorov,widder1976heat}), to analyze some population growth  in biology 
(e.g., \cite{kimura1957some,logan2009mathematical})
to model the evolution of stock prices in finance and economics
(e.g., \cite{barles1997convergence,black1973pricing,wilmott1993option}),
or for climate modeling
(e.g.,\cite{hasselmann1976stochastic,thuburn2005climate}),
to name but a few.

Consider the following\footnote{The choice of the boundary $v(T,x)=f(x)$ to be defined at terminal time is arbitrary and could be placed at initial time using the time change $t\mapsto T-t.$} Kolmogorov PDE\;(see, e.g.,\;\cite{bogachev2022fokker,cerrai2001second})
\begin{align}\label{eq:linear.PDE.intro}
	\left\{
	\begin{aligned}
		&{\partial_t}v(t,x) +   \langle b, \nabla_x v(t,x)\rangle+\frac{1}{2}\operatorname{tr}(\sigma \sigma^\top D_x^2v(t,x))=0\quad\mbox{on}\;\;{\cal D}_{T};\;\\
		&v(T,x)=f(x)\quad\mbox{on}\;\;\mathbb{R}^d.
	\end{aligned}
	\right.
\end{align} 
with ${\cal D}_{T}:=[0,T)\times \mathbb{R}^d$. One of the common modeling challenges arising throughout all fields consists in finding the true drift and volatility parameters ($b$, $\sigma$) to describe the underlying evolution process, which is usually unknown. Typically, one would either try to estimate the parameters using historical data or choose them based on experts' opinions. However, it is well-known that model misspecification may lead to wrong outcomes which might be fatal, as e.g., happened during the financial crisis in 2008 when financial derivatives were priced based on solutions of \eqref{eq:linear.PDE.intro} but with corresponding parameters which were not consistent with the market behavior. 

To overcome this difficulty of \textit{model uncertainty}, a common approach is to consider a set $\mathcal{U}$ of parameters $(b,\sigma)$, where each element $(b,\sigma)\in \mathcal{U}$ is considered as a candidate for the true but unknown drift and volatility. Then, one uses this set of candidates $\mathcal{U}$ to describe the evolution of the underlying process \textit{robustly with respect to its parameters} by considering the following nonlinear Kolmogorov PDE~{(see, e.g.,~\cite{fleming2006controlled,pham2009continuous,yong2012stochastic})}
\begin{align}\label{eq:nonlinear.PDE.intro}
	\left\{
	\begin{aligned}
		&{\partial_t}v(t,x) +  \sup_{(b,\sigma)\in {\cal U}} \left\{ \langle b, \nabla_x v(t,x)\rangle+\frac{1}{2}\operatorname{tr}(\sigma \sigma^\top D_x^2v(t,x))\right\}=0\quad\mbox{on}\;\;{\cal D}_T;\\ 
		&v(T,x)=f(x)\quad\mbox{on}\;\; \mathbb{R}^d.
	\end{aligned}
	\right.
\end{align}

A natural choice for $\mathcal{U}$ we consider throughout this paper is to start with baseline parameters $(b^o,\sigma^o)$ that one considers as first best estimates for the true but unknown drift and volatility and then to consider the set 
\begin{align}\label{eq:uncertain.char}
	\mathcal{B}^\varepsilon:= \big\{(b,\sigma)\in \mathbb{R}^d\times \mathbb{R}^{d\times d}\;:\;\lvert b-b^o \rvert\leq \gamma\varepsilon,\;\lVert \sigma-\sigma^o \rVert_{\operatorname{F}}\leq \eta\varepsilon\big\}
\end{align}
of all coefficients that fall within the (weighted by $\gamma$, $\eta \in [0,1]$) $\varepsilon$-neighborhood of the baseline coefficients, for some pre-specified $\varepsilon>0$. Typical choices for $\gamma$ and $\eta$ consist of 
$(\gamma,\eta)=(1,0)$ representing drift uncertainty \cite{CHMP2002,PengS_gexp1997,PW22,PW23}, $(\gamma,\eta)=(0,1)$ representing volatility uncertainty \cite{DenisHuPeng2011,Nutz2013,Peng07,soner2011quasi,peng2010nonlinear},
as well as $(\gamma,\eta)=(1,1)$ for simultaneous drift and volatility uncertainty~\cite{NeufeldNutz2014,neufeld2017nonlinear,neufeld2018robust,park2025irreversible}. We also refer to \cite{cheridito2007second,matoussi2015robust,soner2012wellposedness,soner2013dual} for the connection of these nonlinear Kolmogorov PDEs~\eqref{eq:nonlinear.PDE.intro} with second-order backward stochastic differential equations.

\vspace{0.2cm}
\noindent
The goal of this paper is to analyze how sensitive  Kolmogorov equations are with respect to their parameters $b$ and $\sigma$.
%
More precisely, for small $\varepsilon>0$ let $v^{\varepsilon}(t,x)$ denote the (unique viscosity) solution of the nonlinear Kolmogorov PDE \eqref{eq:nonlinear.PDE.intro} with $\mathcal{U}:=\mathcal{B}^\varepsilon$ and let $v^{0}(t,x)$ be the solution of the linear one \eqref{eq:linear.PDE.intro} with respect to the baseline parameters $b^o$ and $\sigma^o$. 
In this context, we aim to answer the following~questions:
\vspace{0.1em}

\begin{itemize}[leftmargin=2.em]
	\item [$\cdot$] Can we identify and efficiently calculate the sensitivity $\partial_\varepsilon v^0(t,x):=\lim_{\varepsilon\downarrow 0}\frac{1}{\varepsilon}(v^\varepsilon(t,x)$ $-v^0(t,x))$?
	\item[$\cdot$] Can we find a numerical method which can efficiently solve high-dimensional nonlinear Kolmogorov PDEs of the form \eqref{eq:nonlinear.PDE.intro} with $\mathcal{U}:=\mathcal{B}^\varepsilon$ for  small $\varepsilon>0$?
\end{itemize}


In Theorem~\ref{thm:main} we show that if $f$ is sufficiently regular and satisfies some mild growth conditions (see Assumption~\ref{as:objective}) as well as $\sigma^o$ is invertible, then the following hold. 
For every {$(t,x)\in [0,T)\times \mathbb{R}^d$}, as $\varepsilon\downarrow 0$, we obtain that 
\[
v^\varepsilon(t,x)= v^0(t,x)+ \varepsilon \cdot \partial_\varepsilon v^0(t,x)+O(\varepsilon^2),
\]
where  $\partial_\varepsilon v^0(t,x)$ is  given by
\begin{align*}
	\begin{aligned}
		\partial_{\varepsilon}v^0(t,x) 
		&= \mathbb{E}\bigg[\int_t^T   \gamma \left|w\left(s, x+ X_s^{o}\right)\right|+\eta \left\| \mathrm{J}_xw\left(s, x+X_s^{o} \right) \sigma^o \right\|_{\operatorname{F}}  ds \, \bigg| \, X^o_t= 0 \bigg].
	\end{aligned}
\end{align*}
Here 
\begin{itemize}[leftmargin=2.em]
	\item[--] $X^o_s:=b^o s + \sigma^o W_s$, $s\in[0,T]$ is a stochastic process driven by  a standard $d$-dimensional Brownian motion  $(W_s)_{s\in[0,T]}$,
	\item[--]  $\mathbb{E}[ \,\cdot \, | X^o_t= 0]$ denotes the conditional expectation given $X^o_t= 0$,
	\item[--]  $w=(w^1,\dots,w^d)$ where each  $w^i$ is the solution of the linear  Kolmogorov equation 
	\begin{align*}
		\left\{
		\begin{aligned}
			&\partial_s w^i(s,x)+ \langle b^o, \nabla_x w^i(s,x) \rangle+  \frac{1}{2} \operatorname{tr} \left((\sigma^o) (\sigma^o)^\top D^2_{x}w^i(s,x)\right)  =0 \quad \mbox{on}\;\; {\cal D}_{t,T};\\
			&w^i (T,x)=\partial_{x_i}f(x)\quad  \mbox{on}\;\; \mathbb{R}^d,
		\end{aligned}\right.
	\end{align*}
	with ${\cal D}_{t,T}:=[t,T)\times \mathbb{R}^d$, and $\mathrm{J}_x w$ stands for the Jacobian of  $w$.
\end{itemize} 

We highlight that Theorem~\ref{thm:main} also provides a methodology to approximate the solution $v^{\varepsilon}$ of the nonlinear Kolmogorov PDE \eqref{eq:nonlinear.PDE.intro}. Indeed, note that by the Feynman-Kac representation, we have for any $t\leq s\leq T$ and $x\in \mathbb{R}^d$ that
\begin{equation*}
	\begin{split}
		v^0(t,x) &= \mathbb{E}\Big[f(x+X_T^o)\Big| X_t^o = 0\Big],\\
		w(s,x+X_s^o) &= {\mathbb{E}}\Big[\nabla_xf(x+X_s^o+\widetilde{X}_T^o) \Big| \widetilde{X}_{s}^o = 0 \Big],\\
		(\mathrm{J}_xw)(s,x+X_s^o) &=  {\mathbb{E}}\left[ D^2_x f(x+X_s^o+\widetilde{X}_T^o) \Big| \widetilde{X}_{s}^o = 0 \right],	
	\end{split}
\end{equation*}
where $\widetilde X^o_s:=b^o s + \sigma^o \widetilde W_s$, $s\in[0,T]$, with  $(\widetilde{W}_s)_{s\in[0,T]}$ being another standard $d$-dimensional Brownian motion independent of $(W_s)_{s\in[0,T]}$.
Therefore, we can implement the approximation $v^0+\varepsilon \cdot\partial_\varepsilon v^0$ of $v^\varepsilon$ by a Monte Carlo based scheme (see Algorithm \ref{alg:ours}) which is efficient even in high dimensions (see Section~\ref{sec:numeric} for our numerical results in up to $d=100$ dimensions). We provide an error and computational complexity analysis of our numerical scheme (see Theorem \ref{thm:MC}). Moreover, we present numerical results to support the applicability of our Monte Carlo based scheme.

		\vspace{0.5em}
		\noindent
		\textbf{Related Literature.} Since solutions of Kolmogorov PDEs and parabolic PDEs in general typically cannot be solved explicitly and hence need to be approximately solved, there has been a lot of efforts to develop such numerical approximation methods. We refer~e.g.~to \cite{smolyak1963quadrature,tadmor2012review,thomee2007galerkin} for deterministic approximation methods (e.g., finite difference and finite element methods, spectral Galerkin methods, and sparse grid methods) and to \cite{giles2008multilevel,gobet2016monte,graham2013stochastic} for stochastic approximation methods including Monte Carlo approximations. 
		Recently, there has been an intensive interest in deep-learning based algorithms that can approximately solve high-dimensional linear/nonlinear parabolic PDEs (e.g., \cite{beck2021deep,han2018solving,hure2020deep,sirignano2018dgm}).

		Sensitivity analysis of robust optimization problems have been established mostly with respect to `Wasserstein-type' of uncertainty by considering an (adapted) Wasserstein ball with radius $\varepsilon$  around an (estimated) baseline probability measure for the underlying process either in a one-period model \cite{bartl2021sensitivity,blanchet2019quantifying,fuhrmann2023wasserstein,gao2022wasserstein,gao2023distributionally,mohajerin2018data,nendel2022parametric,obloj2021distributionally,pflug2007ambiguity} or in a multi-period discrete-time  model \cite{bartl2023sensitivity,sauldubois2024first, yifan}. 
        Moreover, in continuous time, 
        \cite{herrmann2017model,herrmann2017hedging} provided a sensitivity analysis of a particular robust utility maximization problem under volatility uncertainty, whereas \cite{BNP2023sensitivity} analyzed the sensitivity of general robust optimization problems under both drift and volatility uncertainty.

		The contribution of our paper is to provide a sensitivity analysis of nonlinear PDEs of type\;\eqref{eq:nonlinear.PDE.intro} and use this analysis to approximate those PDEs by some suitable linear PDEs as described above, leading to a numerical approximation algorithm 
		which is efficient even in high-dimensions.

		\section{Main results}
		\label{sec:main}
		Fix $d\in \mathbb{N}$ and endow $\mathbb{R}^d$ with the Euclidean inner product $\langle \cdot,\cdot\rangle$, and $\mathbb{R}^{d\times d}$ with the Frobenius inner product $\langle \cdot,\cdot \rangle_{\operatorname{F}}$, respectively. Let $\mathbb{S}^d$ be the set of all symmetric $d\times d$ matrices. Then fix a time horizon $T>0$, and for any $\varepsilon\geq0$ consider a nonlinear Kolmogorov PDE with the set $\mathcal{B}^\varepsilon$ given in \eqref{eq:uncertain.char}
		\begin{align}\label{eq:nonlinear.pde}
			\left\{
			\begin{aligned}
				&\partial_tv^\varepsilon(t,x)+ \sup_{(b,\sigma)\in \mathcal{B}^\varepsilon}\left\{\langle b, \nabla_x v^\varepsilon (t,x) \rangle  +\frac{1}{2} \operatorname{tr} \big(\sigma \sigma^\top D^2_{x}v^\varepsilon(t,x) \big) \right\}=0\quad \mbox{on}\;\; {\cal D}_T;\\
				&v^\varepsilon(T,x) = f(x)\quad \mbox{on} \;\; \mathbb{R}^d,
			\end{aligned}
			\right.
		\end{align}
		with ${\cal D}_{T}=[0,T)\times \mathbb{R}^d$, where $f:\mathbb{R}^d\rightarrow \mathbb{R}$ corresponds to the boundary condition. 
		
		We impose certain conditions on the boundary $f$ and coefficient $\sigma^o$ given in \eqref{eq:uncertain.char}. 
		\begin{as}\label{as:objective}
			The function $f\colon\mathbb{R}^d\to\mathbb{R}$ is continuously differentiable.
			Moreover, its Hessian $D_x^2f:\mathbb{R}^d\rightarrow \mathbb{S}^d$ exists in the weak sense\footnote{We refer to, e.g., \cite[Section 5.2]{evans2010partial} for the definition of weak derivatives.} and there are $\alpha \geq 1$ and $C_f>0$ such that  $\| D_x^2 f(x) \|_{\rm F} \leq C_f (1+|x|^\alpha)$ for every $x\in\mathbb{R}^d$.
		\end{as}
		
		\begin{as}\label{as:sigma.inverse} 
			The matrix $\sigma^o$ is invertible.
		\end{as}
		
		\begin{rem}\label{rem:sigma.inverse}
			Assumption \ref{as:sigma.inverse} ensures that $\lambda_{\min}(\sigma^o)$, the smallest singular value  of $\sigma^o$ is strictly positive; in particular for every $\varepsilon<\lambda_{\min}(\sigma^o)$ and  $(b,\sigma)\in\mathcal{B}^\varepsilon$, the matrix  $\sigma$ is invertible.
		\end{rem}
		
		We further impose a condition on the solution of the nonlinear Kolmogorov PDE, which relies on the notion of viscosity solutions (see Section \ref{sec:proof:pro:main} for the standard definitions of viscosity / classical solutions of PDEs).
		
		\begin{as}\label{as:comparison}
			For any $\varepsilon\geq 0$, there exists at most one viscosity solution $v^\varepsilon$ of \eqref{eq:nonlinear.pde} satisfying that there is $C>0$ such that for all $t\in[0,T]$ 
			\begin{align}\label{eq:growth_as}
				\lim_{|x|\rightarrow\infty}|v^\varepsilon(t,x)|e^{-C(\log(|x|))^2}=0.
			\end{align}
		\end{as}
		
		\begin{rem}\label{rem:comparison}
			It follows from \cite[Proposition 5.5]{neufeld2017nonlinear} that Assumption \ref{as:comparison} is satisfied if, e.g.,  Assumptions \ref{as:objective} and \ref{as:sigma.inverse} are satisfied and $f$ is bounded and Lipschitz continuous. Furthermore, if $\eta=0$,\;i.e.,\;there is no volatility uncertainty, then Assumptions\;\ref{as:objective}\;\&\;\ref{as:sigma.inverse} directly imply that Assumption~\ref{as:comparison} holds, see \cite[Theorem\;3.5]{barles1997backward}. We also refer to \cite[Remark\;3.6]{barles1997backward} for a detailed discussion on the growth condition\;\eqref{eq:growth_as}. 
		\end{rem}
		
		Now we collect some preliminary results in the next proposition on the solution of the nonlinear Kolmogorov PDE \eqref{eq:nonlinear.pde} together with the following linear Kolmogorov PDE defined using the baseline coefficients $b^o$ and $\sigma^o$ and the boundary condition $\partial_{x_i}f$, $i=1,\dots,d$, 
		\begin{align}\label{eq:linear.pde.1stdev}
			\left\{
			\begin{aligned}
				&\partial_s w^i(s,x)+ \langle b^o, \nabla_x w^i(s,x) \rangle+  \frac{1}{2} \operatorname{tr} \left((\sigma^o) (\sigma^o)^\top D^2_{x}w^i(s,x)\right)  =0\quad \mbox{on}\;\;{\cal D}_{t,T};\\
				&w^i (T,x)=\partial_{x_i}f(x)\quad \mbox{on}\;\; \mathbb{R}^d,
			\end{aligned}\right.
		\end{align}
		with ${\cal D}_{t,T}=[t,T)\times \mathbb{R}^d$, where we note that $f$ is the boundary condition given in~\eqref{eq:nonlinear.pde}.

		The corresponding proofs for the following proposition 
		can be found in Section~\ref{sec:proof:pro:main}. 
		\begin{pro}\label{pro:main}
			Suppose that Assumptions \ref{as:objective}, \ref{as:sigma.inverse}, and \ref{as:comparison} 
			are satisfied. Then the following hold: 
			\begin{itemize}
				\item [(i)] For any $\varepsilon\geq 0$, there exists a unique viscosity solution $v^\varepsilon:[0,T]\times\mathbb{R}^d\rightarrow \mathbb{R}$ of \eqref{eq:nonlinear.pde} satisfying the growth property given in \eqref{eq:growth_as}. 
				\item [(ii)] For any $i=1,\dots,d$, there exists a unique classical solution $w^i:[t,T]\times\mathbb{R}^d\rightarrow \mathbb{R}$ of \eqref{eq:linear.pde.1stdev} with polynomial growth. 
			\end{itemize}
		\end{pro}
		We proceed with our main result.
		To formulate it, denote by $O(\cdot)$ the Landau symbol and if $w^i$ is the solution to \eqref{eq:linear.pde.1stdev} with the boundary condition $\partial_{x_i}f$ for every $i=1,\dots, d$, let us define $w: [t,T]\times \mathbb{R}^d\rightarrow \mathbb{R}^d$ and $\mathrm{J}_x w: [t,T]\times \mathbb{R}^d\rightarrow \mathbb{R}^{d\times d}$ by 
		\begin{align}\label{eq:vector.value.jacobi}
			\begin{aligned}
				w(s,x) := \begin{pmatrix}
					w^1(s, x)\\
					\vdots\\
					w^d(s, x)
				\end{pmatrix},
				\qquad \mathrm{J}_xw (s,x) := 
				\begin{pmatrix}
					\nabla_x^\top  w^1(s, x) \\
					\vdots\\
					\nabla_x^\top  w^d(s, x)
				\end{pmatrix}.
			\end{aligned}
		\end{align}
        Note that under Assumptions \ref{as:objective} and \ref{as:sigma.inverse}, $v^0$ is the unique classical solution of the PDE~\eqref{eq:linear.PDE.intro} with coefficients $(b,\sigma)$ replaced by the reference coefficients $(b^o,\sigma^o)$; see, e.g., \cite{KS1991,friedman1975stochastic,Krylov1999}. 
       In particular, we have that $w=\nabla_x v^0$ and $\mathrm{J}_xw=D_x^2v^0$ on ${\cal D}_{t,T}$.
        
		Finally, let $X^o_t=b^o t + \sigma^o W_t$, $t\in[0,T]$, be defined on a probability space $(\Omega,{\cal F},\mathbb{P})$ with respect to a fixed $d$-dimensional Brownian motion $W=(W_t)_{t\in[0,T]}$. 
		
		\begin{thm}\label{thm:main}
			Suppose that Assumptions \ref{as:objective}, \ref{as:sigma.inverse}, and \ref{as:comparison} are satisfied. 
			For every $\varepsilon\geq 0$, let  $v^\varepsilon$ be the unique viscosity solution of \eqref{eq:nonlinear.pde} satisfying \eqref{eq:growth_as}, let $w^i$ be the unique classical solution of \eqref{eq:linear.pde.1stdev} with polynomial growth for every $i=1,\dots, d$, and let $w$,~$\mathrm{J}_xw$ be as in 
			\eqref{eq:vector.value.jacobi}.
			Then, for every $(t,x)\in [0,T)\times \mathbb{R}^d$ as $\varepsilon\downarrow 0$,
			\[
			v^\varepsilon(t,x)= v^0(t,x)+ \varepsilon \cdot \partial_\varepsilon v^0(t,x)+O(\varepsilon^2),
			\]
			where  $\partial_\varepsilon v^0(t,x)=\lim_{\varepsilon\downarrow 0}\frac{1}{\varepsilon}(v^\varepsilon(t,x)-v^0(t,x))$ is  given by
			\begin{align*}
				\begin{aligned}
					\partial_{\varepsilon}v^0(t,x) 
					&= \mathbb{E}\bigg[\int_t^T \left(  \gamma \left|w\left(s, x+ X_s^{o}\right)\right|+\eta \left\| \mathrm{J}_xw\left(s, x+X_s^{o} \right) \sigma^o \right\|_{\operatorname{F}} \right) ds \, \bigg| \, X^o_t= 0 \bigg],
				\end{aligned}
			\end{align*}
			with $\mathbb{E}[\cdot | X^o_t= 0]$ denoting the conditional expectation given $X^o_t= 0$.
		\end{thm}
		
		\begin{rem}\label{rem:main_thm}
			We actually show that the approximation is (locally) uniform  in  $(t,x)$:
			There exists a constant $c$ (that depends on $T$, $\alpha$, $d$ and $C_f$ given in Assumption \ref{as:objective}, and the norms for $b^o$,\;$\sigma^o$) such that for every $\varepsilon < \min\{1,\lambda_{\min}(\sigma^o)\}$ (see Remark \ref{rem:sigma.inverse}) and every $(t,x)\in[0,T)\times\mathbb{R}^d$,
			\[ \left|	v^\varepsilon(t,x) - \left(  v^0(t,x)+ \varepsilon \cdot \partial_\varepsilon v^0(t,x) \right) \right| 
			\leq c (1+|x|^\alpha) \varepsilon^2.
			\]
		\end{rem}
		
		\vspace{0.5em}
		Let us mention some basic properties of the sensitivity result given in Theorem\;\ref{thm:main}, as well as how it can be used to construct numerical approximations of the PDE \eqref{eq:nonlinear.pde}, as explained in Section \ref{sec:numeric} below.
		To that end, recalling the process $X^o$ appearing in Theorem \ref{thm:main} with the corresponding Brownian motion $W$, let $\widetilde{X}^o_t:=b^o t + \sigma^o \widetilde{W}_t,$ $t\in[0,T]$, where $\widetilde{W}:=(\widetilde{W}_t)_{t\in[0,T]}$ is another standard $d$-dimensional Brownian motion independent of $W$. We remark the following Feynman-Kac representations:
		for every $(t,x)\in[0,T]\times \mathbb{R}^d$ and $s\in [t,T]$, it holds that 
		\begin{align}\label{eq:Feyn_1}
			\begin{aligned}
				v^0(t,x) &= \mathbb{E}\Big[f(x+X_T^o)\Big| X_t^o = 0\Big],\\
				w(s,x+X_s^o) &= {\mathbb{E}}\Big[\nabla_xf(x+X_s^o+\widetilde{X}_T^o) \Big| \widetilde{X}_{s}^o = 0 \Big].
			\end{aligned} 
		\end{align}
		
		Furthermore, denote by $(\mathrm{J}_xw)^{k,l}$ for every $k,l\in \{1,\dots,d\}$ the $(k,l)$-component of $\mathrm{J}_xw$ defined in \eqref{eq:vector.value.jacobi}. If $\nabla_x f$ is sufficiently smooth (at least continuously differentiable) and $D^2_xf$ is at most polynomially growing, then by \cite[Theorem~4.32]{Krylov1999}, $(\mathrm{J}_xw)^{k,l}$ also has the following Feynman-Kac representation: for $(t,x)\in [0,T]\times \mathbb{R}^d$ and $s\in[t,T]$,
		\begin{align}\label{eq:Feyn_2}
			\begin{aligned}
				(\mathrm{J}_xw)^{k,l}(s,x+X_s^o) =  {\mathbb{E}}\left[ \partial_{x_{k} x_{l}} f(x+X_s^o+\widetilde{X}_T^o) \Big| \widetilde{X}_{s}^o = 0 \right]. 
			\end{aligned}
		\end{align}
		Otherwise, if $\nabla_xf$ lacks that kind of regularity, we approximate $(\mathrm{J}_xw)^{k,l}$  via a finite difference quotient as follows: Let $e_l$ be a $d$-dimensional vector with value 0 in all the components except for the $l$-th component with value 1. Then 
		for sufficiently small~$h>0$, 
		\[
		(\mathrm{J}_xw)^{k,l} (s,x+X_s^o) = \partial_{x_l}w^k (s,x+X_s^o)  \approx \frac{1}{h}({w}^k(s,x+X_s^o+h\cdot e_l)-{w}^k(s,x+X_s^o)).
		\] 
		
		In particular, by \eqref{eq:Feyn_1} the approximation can be rewritten by 
		\begin{align}\label{eq:Feyn_3}
			\begin{aligned}
				&(\mathrm{J}_xw)^{k,l} (s,x+X_s^o)\\
				&\quad \approx \frac{1}{h}{\mathbb{E}}\left[\partial_{x_k}f(x+X_s^o+\widetilde{X}_T^o+h\cdot e_l )-\partial_{x_k}f(x+X_s^o+\widetilde{X}_T^o) \Big| \widetilde{X}_{s}^o = 0\right].
			\end{aligned}
		\end{align}
		Hence, the exact value of $\partial_\varepsilon v^0$ requires calculations of nested expectations. 

		\begin{algorithm}[t]
			\caption{ A Monte Carlo based scheme for $v^0$ and $\partial_{\varepsilon}v^0$.}\label{alg:ours}
			\begin{algorithmic}[1]
				{\footnotesize
					\STATE \textbf{Input:} $T>0$, $b^o\in \mathbb{R}^d$, $\sigma^o\in \mathbb{R}^{d\times d}$, $f:\mathbb{R}^d\rightarrow \mathbb{R}$ (satisfying Assumptions \ref{as:objective}, \ref{as:sigma.inverse}, and \ref{as:comparison}), $(t,x)\in [0,T)\times\mathbb{R}^d$, $N\in \mathbb{N}$, $M_0,M_1,M_2\in \mathbb{N}$ with $M_0\geq M_1$ and $h\geq 0$ (sufficiently~small); \\
					\STATE \textbf{Generate:}  \\
					\hskip1.5em (i) The uniform subdivision $t_i=t+i \Delta t$, $i\in \{0,\dots,N\}$ with $\Delta t=\frac{T-t}{N}$; \\
					\hskip1.5em (ii) $M_0$ samples ${\cal W}(j)\sim {\cal N}(0,\mathrm{Id}_{\mathbb{R}^d})$ (i.e., standard $d$-dimensional normal distribution), $j\in \{1,\dots,M_0\}$, and $M_2$ samples $\widetilde{{\cal W}}(m)\sim {\cal N}(0,\mathrm{Id}_{\mathbb{R}^d})$, $m\in\{1,\dots,M_2\}$; \\
					\hskip1.5em (iii) $(N+1)\times M_0$ samples ${\cal X}_i(j):=b^o t_i +\sigma^o {\cal W}(j) \sqrt{t_i}$, $i\in \{0,\dots,N\}$ and $j\in \{1,\dots,M_0\}$ and $(N+1)\times M_2$ samples $\widetilde {\cal X}_i(m):=b^o t_i +\sigma^o \widetilde {\cal W}(m) \sqrt{t_i}$, $i\in \{0,\dots,N\}$ and $m\in \{1,\dots,M_2\}$;\\
					\STATE \textbf{Function} $v^0_{\operatorname{mc}}(t,x;M_0)$\textbf{:} \\ 
					\hskip1.5em Recall $M_0$ realizations ${\cal X}_N(j)$, $j\in \{1,\dots,M_0\}$; \\
					\textbf{Return} {$\frac{1}{M_0} \sum_{j=1}^{M_0} f(x+{\cal X}_N(j))$}\\  
					\STATE \textbf{Function} $(\partial_\varepsilon v^0)_{\operatorname{mc}}(t,x;N,M_1,M_2)$\textbf{:} \\ 
					\hskip1.5em Recall $N\times M_1$ realizations ${\cal X}_i(j)$, $i\in \{0,\dots,N-1\}$ and $j\in \{1,\dots,M_1\}$, and $N\times M_2$ realizations $\widetilde {\cal X}_i(m)$, $i\in \{0,\dots,N-1\}$ and $m\in \{1,\dots,M_2\}$; \\
					\hskip1.5em {\bf for} $i=0$ {\bf to} $N-1$ {\bf \textrm{and}} $j=1$ {\bf to} $M_1$ \\
					\hskip3em  Compute $\widehat{w}^k(i,j):=\frac{1}{M_2} \sum_{m=1}^{M_2} \partial _{x_k} f(x+{\cal X}_i(j)+\widetilde {\cal X}_{N-i}(m))$ $\forall k\in \{1,\dots,d\}$; \\  
					\hskip1.5em {\bf end}\\
					\hskip1.5em {\bf if} {$\nabla_x f$ {\normalfont is continuously differentiable and has at most polynomial growth}} \\
					\hskip3.em {\bf for} $i=0$ {\bf to} $N-1$ {\bf and} $j=1$ {\bf to} $M_1$ \\
					\hskip4.5em Compute $(\widehat{\mathrm{J}_xw})^{k,l}(i,j)  :=\frac{1}{M_2} \sum_{m=1}^{M_2} \partial_{x_{k},x_{l}} f(x+{\cal X}_i(j)+\widetilde {\cal X}_{N-i}(m))$ $\forall k,l\in \{1,\dots,d\}$;\\
					\hskip3.em {\bf end} \\
					\hskip 1.5em {\bf else} \\ 
					\hskip3.em {\bf for} $i=0$ {\bf to} $N-1$ {\bf and} $j=1$ {\bf to} $M_1$ \\
					\hskip4.5em Compute $\widehat{w}^{k,l}_h(i,j)  :=\frac{1}{M_2} \sum_{m=1}^{M_2} \partial_{x_{k}} f(x+{\cal X}_i(j)+he_l+\widetilde{\cal X}_{N-i}(m))$  $\forall k,l\in \{1,\dots,d\}$; \\
					\hskip4.5em Compute $(\widehat{\mathrm{J}_xw})^{k,l}(i,j)  :=\frac{1}{h} (\widehat{w}^{k,l}_h(i,j) -\widehat{w}^k(i,j))$ $\forall k,l\in \{1,\dots,d\}$;\\
					\hskip3.em {\bf end} \\
					\hskip1.5em {\bf end} \\
					{\bf Return}{ $\sum_{i=0}^{N-1}  \Delta t \frac{1}{M_1} \sum_{j=1}^{M_1} (\gamma \big|\widehat{w}(i,j)\big| + \eta \| \widehat{\mathrm{J}_xw}(i,j) \sigma^o \|_{\operatorname{F}} )$}
				}
			\end{algorithmic}
		\end{algorithm}

			Combining Theorem \ref{thm:main} with corresponding probabilistic representations for the functions $v^0$, $w$, and $\mathrm{J}_xw$  given in \eqref{eq:Feyn_1} and \eqref{eq:Feyn_2} (or \eqref{eq:Feyn_3}), we derive a Monte Carlo based scheme to implement both the sensitivity $\partial_\varepsilon v^0$ as well as the approximated solution  $v^0+\varepsilon \cdot \partial_\varepsilon v^0$ of the nonlinear PDE \eqref{eq:nonlinear.pde} for every $\varepsilon< \min\{1,\lambda_{\min}(\sigma^o)\}$ (see Remark \ref{rem:main_thm}). 	We provide a pseudo-code in Algorithm~\ref{alg:ours} to show how it can be implemented\footnote{\label{footnote:link}All the numerical experiments have been performed with the following hardware configurations: a {{Macbook Pro}} with {{Apple M2 Max}} chip, 32 GBytes of memory, and Mac OS 13.2.1. While we implement the \textsc{Matlab} code only on the CPU, we take advantage of the GPU acceleration (Metal Performance Shaders (MPS) backend) for implementing the \textsc{Python} codes. All the codes are provided in the following link: \url{https://github.com/kyunghyunpark1/Sensitivity_nonlinear_Kolmogorov}}.

			Moreover, we provide an error and complexity analysis of our numerical scheme. To that end, let us introduce some notation.  
			
			Fix $(t,x)\in [0,T)\times \mathbb{R}^d$, and let $N\in \mathbb{N}$ (i.e., the number of steps in the time discretization for $\partial_\varepsilon v^0$) and $M_0,M_1,M_2\in \mathbb{N}$ 
			(i.e., the number of samples for each expectation involved in $v^0$ and $\partial_\varepsilon v^0$) be given.
			\begin{itemize}[leftmargin=2.5em]
				\item [(i)] 
				Let $t_i:=t+i \Delta t$, $i\in \{0,\dots,N\}$ with $\Delta t:=\frac{T-t}{N}$;
				\item [(ii)] For $i\in\{0,\dots,N-1\}$, let $\xi^1_{T,t},\dots,\xi^{M_0}_{T,t}$ be independently identically distributed (i.i.d.) random samples of $M_0$ realizations of $X_T^o$ given $X_t^o=0$. Furthermore, let $\xi_{t_i,t}^1,\dots,\xi_{t_i,t}^{M_1}$ be i.i.d~random samples of $M_1$ realizations of $X_{t_i}^o$ given $X_t^o=0$ and $\widetilde \xi_{T,t_i}^1,\dots,\widetilde \xi_{T,t_i}^{M_2}$ be i.i.d~random samples of $M_2$ realizations of $\widetilde X_{T}^o$ given $\widetilde X_{t_i}^o=0$, where %
				$\xi_{t_i,t}^{m}$ and $\widetilde \xi_{T,t_i}^{n}$ are independent for every $m=1,\dots,M_1$ and $n=1,\dots,M_2$. 
			\end{itemize}
			For a sufficiently integrable random variable $Z$, we set $\|Z\|_{L^2}:=\mathbb{E}[|Z|^2]^\frac{1}{2}$. 
			
			\begin{thm}\label{thm:MC}
				Suppose that Assumptions \ref{as:sigma.inverse} and \ref{as:comparison} are satisfied. 
				Furthermore, assume that $f$ is three times continuously differentiable with its third derivatives having at most  polynomial growth.
				Let $(t,x)\in [0,T)\times \mathbb{R}^d$ and $\varepsilon<1\wedge \lambda_{\min}(\sigma^o)$ be given. Denote for $N,M_0,M_1,M_2\in \mathbb{N}$~by 
				\begin{align}\label{eq:MC_estimates}
					\begin{aligned}
						\mbox{$\widehat{v}^{0,(t,x)}_{M_0}$}&\mbox{$:= \frac{1}{M_0} \sum_{l=1}^{M_0} f(x+\xi_{T,t}^l)$},\\
						\mbox{$\widehat{v}^{1,(t,x)}_{N,M_1,M_2}$}&\mbox{$:= \sum_{i=0}^{N-1}\frac{1}{M_1}\sum_{m=1}^{M_1}|\frac{1}{M_2}\sum_{n=1}^{M_2}\nabla_xf(x+\xi_{t_i,t}^m+\widetilde {\xi}^n_{T,t_i})|\Delta t$},\\
						\mbox{$\widehat{v}^{2,(t,x)}_{N,M_1,M_2}$}&\mbox{$:=\sum_{i=0}^{N-1}\frac{1}{M_1}\sum_{m=1}^{M_1}\|\frac{1}{M_2}\sum_{n=1}^{M_2}D^2_xf(x+\xi_{t_i,t}^m+\widetilde {\xi}^n_{T,t_i})\sigma^o\|_{\operatorname{F}}\Delta t$}.
					\end{aligned}
				\end{align}
				\begin{itemize}
					\item [(i)] Then as $M_0,M_1,M_2,N\rightarrow \infty$,  
					\begin{align*}
						\begin{aligned}
							&\Big\|v^0(t,x) +\varepsilon\cdot \partial_{\varepsilon}v^0(t,x)-\Big(\widehat{v}^{0,(t,x)}_{M_0}  +\varepsilon\cdot\big(\gamma\widehat{v}^{1,(t,x)}_{N,M_1,M_2}+\eta \widehat{v}^{2,(t,x)}_{N,M_1,M_2}\big)\Big)\Big\|_{L^2}\\
							&\quad =O(1/\sqrt{N})+O(1/{\sqrt {M_0}})+O(1/\sqrt{M_1})+O(1/\sqrt{M_2}).
						\end{aligned}
					\end{align*}
					\item [(ii)] Denote by $\mathfrak{C}(N,M_0,M_1,M_2)$ the computational complexity of the estimates given in \eqref{eq:MC_estimates}. Then, as $M_0,M_1,M_2,N\rightarrow \infty$, 
					\begin{align*}
						\mathfrak{C}(N,M_0,M_1,M_2)=  O(M_0) + O(N M_1 M_2).
					\end{align*}
				\end{itemize}
			\end{thm}

			\begin{rem}\label{rem:dim}
				The Landau symbol $O$ in Theorems \ref{thm:main}, \ref{thm:MC} depends on the parameters $t,x,f, T, b, \sigma,d$ through a multiplicative constant. 
				Our proof shows that this dependence is of a low polynomial degree, and in particular not exponential in $d$.
			\end{rem}

			\section{Examples}\label{sec:numeric}
			We analyze two numerical examples to support the applicability of Algorithm \ref{alg:ours}.
		\subsection{Example\;1}\label{sec:ex1}
	Let us start with the following $1$-dimensional example.
Let $d=1$, $T=1$, $b^o= 1$, $\sigma^o=1$, $f(x)=x^4$, $(t,x)=(0,0)$, $N=100$, $M_0=2.4\times 10^6$, and $M_1=M_2= 2.4\times 10^4$. Under this case, Assumptions \ref{as:objective} and \ref{as:sigma.inverse} are obviously satisfied. Furthermore, since $\lambda_{\min}(\sigma^o)=1$, we can and do choose any~$\varepsilon <1$.  

As the function $f=x^4$ is convex, the probabilistic representation for $v^\varepsilon$ given in \eqref{eq:strong_preview} (that will be proven in Lemma \ref{lem:weak_viscosity} and Proposition \ref{pro:equiv.weak.strong})  ensures that $v^\varepsilon(t,x)$ is convex in $x$ and hence the corresponding nonlinear Kolmogorov PDE given in \eqref{eq:nonlinear.pde} can be rewritten by the following quasilinear parabolic equation 
\begin{align}\label{eq:quasi_linear}\partial_tv^\varepsilon+ \frac{ (\sigma^o+\eta\varepsilon)^2}{2} \partial_{xx}v^\varepsilon+{b}^o \partial_xv^\varepsilon+\sup_{|\widetilde{b}|\leq \gamma\varepsilon }\big( \widetilde{b} \partial_xv^\varepsilon\big) =0\quad \mbox{on}\;\;{\cal D}_T;
\end{align} 
and $v^\varepsilon(T,x)=x^4$ on $\mathbb{R}$, with ${\cal D}_T=[0,T)\times \mathbb{R}^d$. In particular, since the PDE \eqref{eq:quasi_linear} is linear in the second derivative and the boundary $f=x^4$ is a polynomial, Remark\;\ref{rem:comparison} guarantees that
\eqref{eq:quasi_linear} admits a unique viscosity solution satisfying \eqref{eq:growth_as}, which ensures Assumption~\ref{as:comparison} to hold.

\begin{figure}[t]
	\centering
	\subfigure[Case when $(\gamma,\eta)=(1,0),\quad \quad\;$  $\quad v^\varepsilon \approx 10.000 + \varepsilon \cdot 16.4908$\;\;with$\quad\quad$  $\quad\quad v^0=10.000$, $\partial_\varepsilon v^0 =16.4908$.]{\label{fig00a}\includegraphics[scale=0.313]{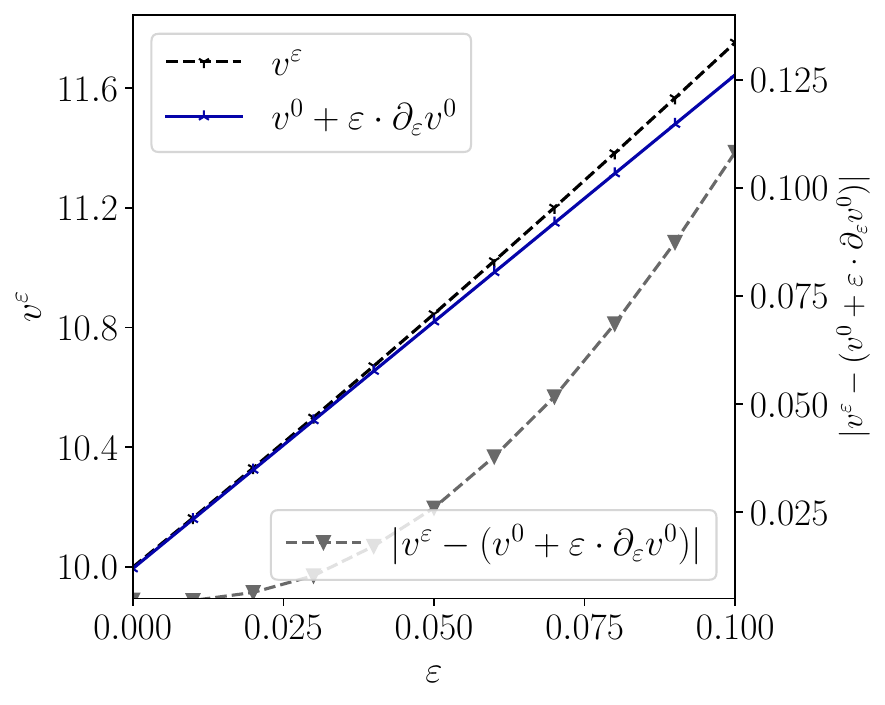}}\;
	\subfigure[Case when $(\gamma,\eta)=(0,1),\quad \quad\;$ $v^\varepsilon \approx 10.000  +\varepsilon \cdot  24.1101$\;\;with$\quad \quad$ $\quad \quad v^0=10.000$, $\partial_\varepsilon v^0 =24.1101$]{\label{fig00b}\includegraphics[scale=0.313]{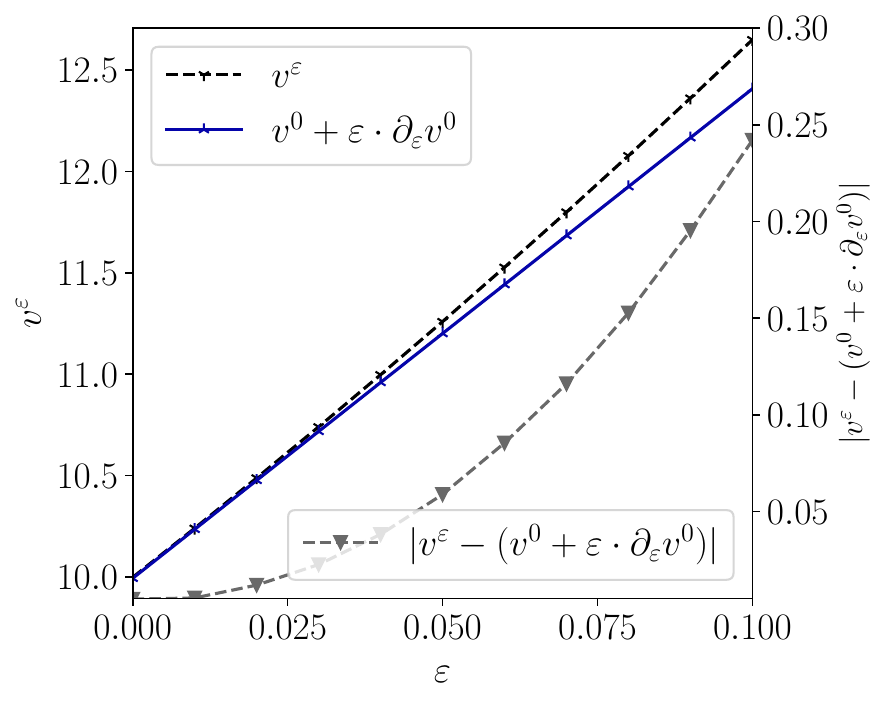}}\;
	\subfigure[Case  when $(\gamma,\eta)=(1,1),\quad \quad\;$  $v^\varepsilon\approx 10.000 + \varepsilon\cdot 40.6098$\;\;with$\quad\quad$ $v^0=10.000$, $\partial_\varepsilon v^0 = 40.6009$.]{\label{fig00c}\includegraphics[scale=0.313]{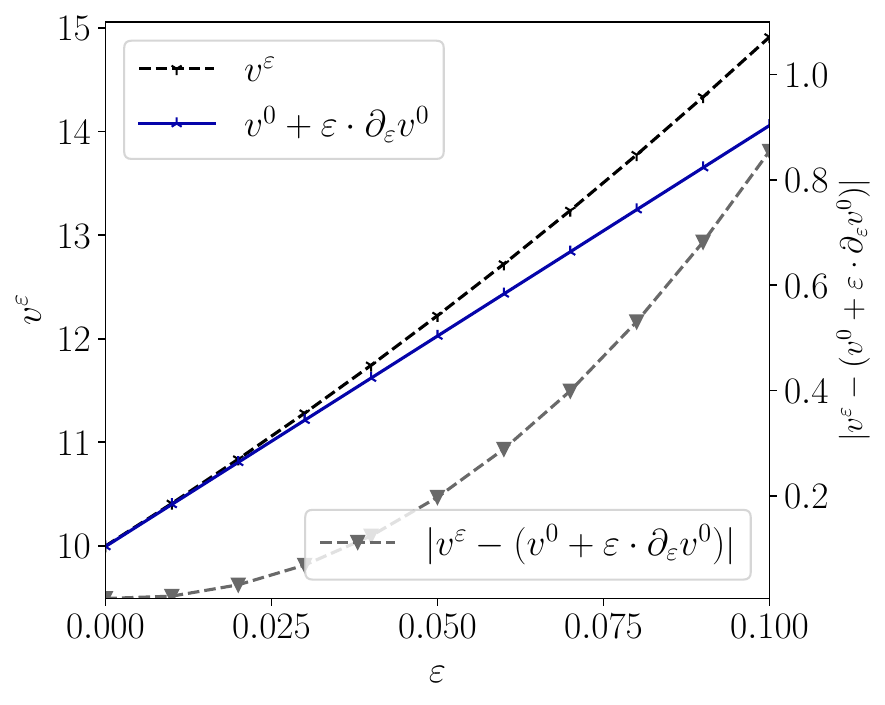}}
	\caption{Comparative analysis between the approximated solution $v^0+\varepsilon\cdot \partial_\varepsilon v^0$ and the actual counterpart $v^\varepsilon$  over varying $\varepsilon$.}
	\label{fig00}
\end{figure}

\begin{figure}[t]
	\centering
	\subfigure[Stability w.r.t.\;$N\quad\quad\;$  in {[$10,\dots,100$]}.]{\label{fig001b}\includegraphics[scale=0.235]{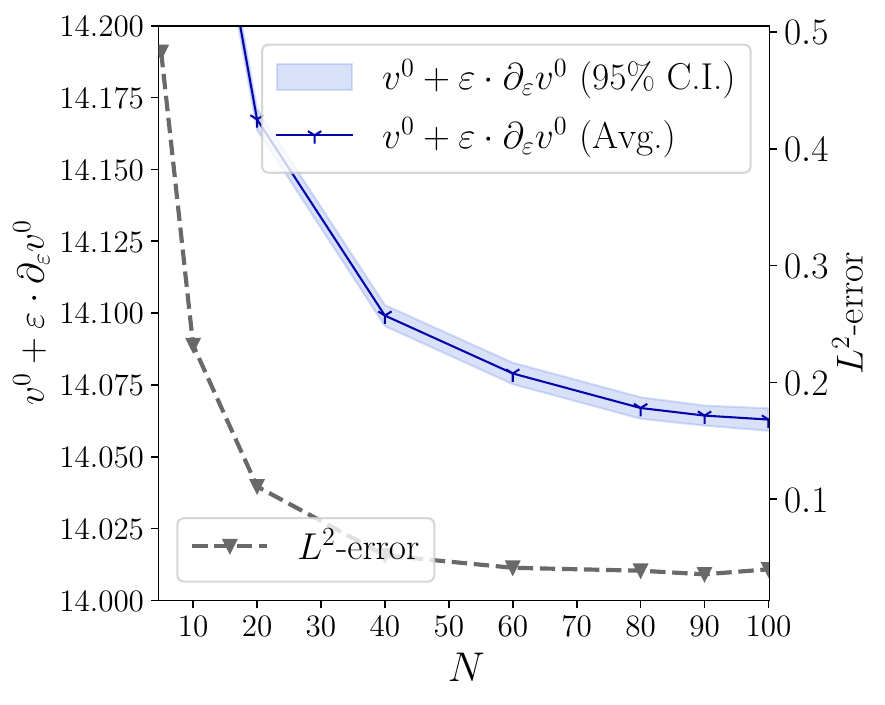}}
	\subfigure[Stability w.r.t.\;$M_0\quad\quad$ in {[$5 \cdot 10^4,\dots, 2.4\cdot 10^6$]}.]{\label{fig001a}\includegraphics[scale=0.235]{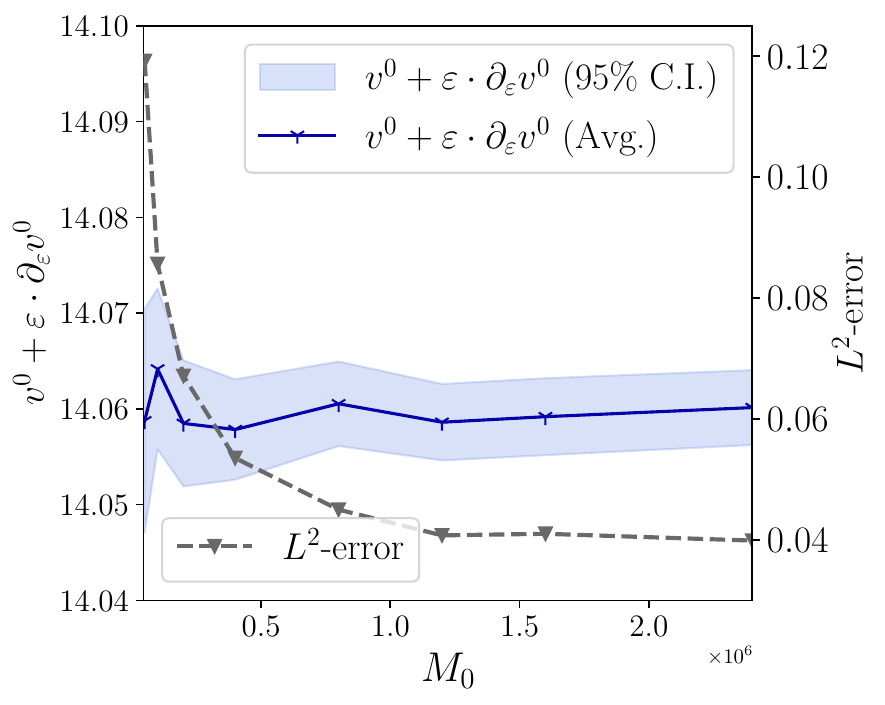}}
	\subfigure[Stability w.r.t.\;$M_1\quad\quad$ in {[$10^3,\dots,2.4\cdot 10^4$]}.]{\label{fig001c}\includegraphics[scale=0.235]{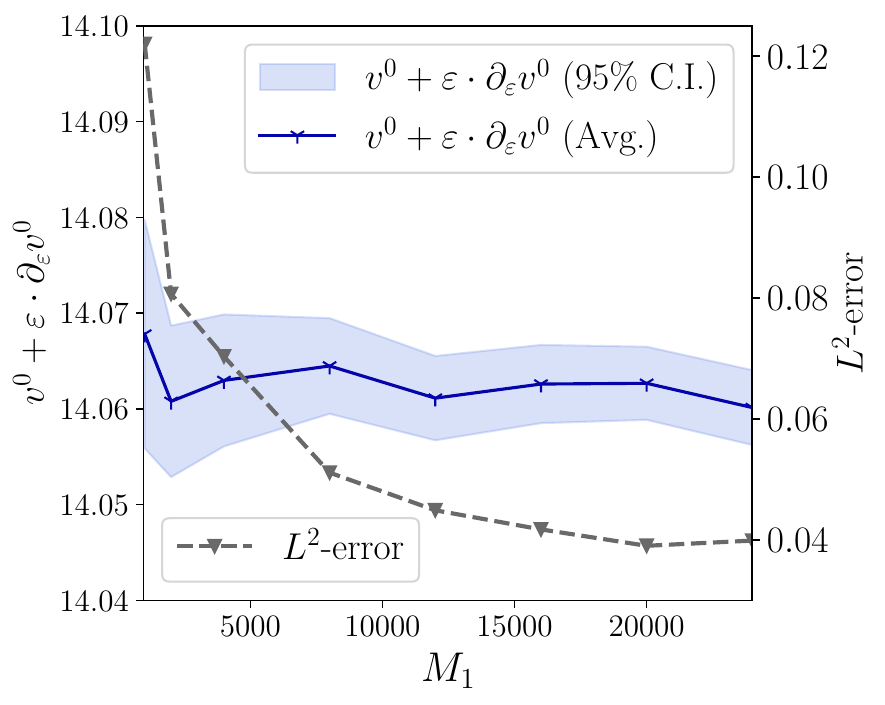}}
	\subfigure[Stability w.r.t.\;$M_2\quad\quad$ in {[$10^3,\dots,2.4\cdot 10^4$]}.]{\label{fig001d}\includegraphics[scale=0.235]{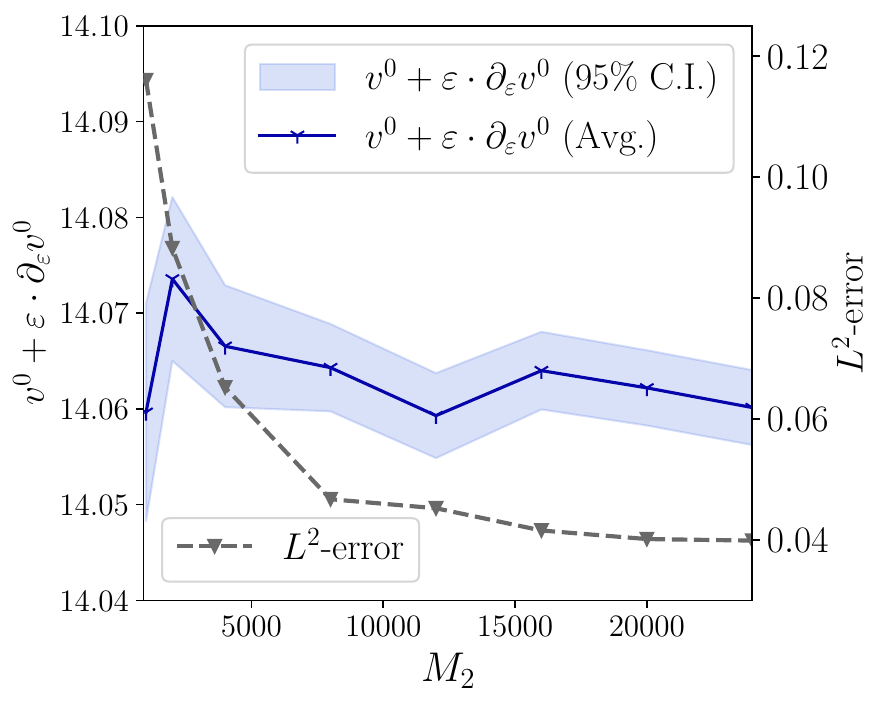}}
	\caption{{\small Stability analysis with respect to the model parameters $N,M_0,M_1,M_2$ in Algorithm~\ref{alg:ours}. 
			We fix $\varepsilon=0.1$ and $\gamma=\eta=1$, and consider the basis parameters $N=100$, $M_0=2.4\times 10^6$ and $M_1=M_2= 2.4\times 10^4$, as specified in Section\;\ref{sec:ex1}. 
			In each plot, all but one parameters are fixed (e.g.\ in the left hand image, $M_0,M_1,M_2$ are fixed and $N$ is varied).
			The plots features blue lines to represent the average of estimates $v^0+\varepsilon \cdot \partial_\varepsilon v^0$ across 400 independent runs of the \textsc{Python} code. Moreover, the grey dashed lines depict the  $L^2$ errors across the 400 runs, where the benchmark for these errors is the average value (i.e., $v^0+\varepsilon \cdot \partial_\varepsilon v^0=14.0601$) under the basis parameters.}}
\label{fig001}
\end{figure}

Fig.\;\ref{fig00} shows the comparison between $v^\varepsilon$ and $v^0+ \varepsilon \cdot \partial_\varepsilon v^0$ with varying $\varepsilon\leq 0.1$, in which we obtain numerical results on $v^\varepsilon$ by applying a finite difference approximation on the semilinear PDE \eqref{eq:quasi_linear} (we refer to {\tt Code{\_}1.m} and {\tt Code{\_}2.m} given in the link provided in Footnote\;\ref{footnote:link}; based on \cite[\textsc{Matlab} Code 7 in Section~3]{hutzenthaler2019multilevel}) and obtain numerical results on $v^0+ \varepsilon \cdot \partial_\varepsilon v^0$ by using our Monte Carlo based scheme given in Algorithm \ref{alg:ours} (we refer to {\tt Code{\_}3.ipynb} given in the mentioned link).  As it has been proven in Theorem \ref{thm:main},  we can observe in all the plots of Figure \ref{fig00} that the error $|v^\varepsilon-(v^0+ \varepsilon \cdot \partial_\varepsilon v^0)|$  of the approximation increases quadratically in $\varepsilon$.

Fig.\;\ref{fig001} shows that the average estimates of $v^0+\varepsilon \cdot \partial_\varepsilon v^0$ only have small variation with respect to the number of Monte Carlo samples $M_0,M_1,M_2$ (as indicated in plots\;(b)-(d)), while they first rapidly converge to and then stabilize at the benchmark estimate as the number of discretization $N$ increases to 100 (see plot\;(a)).  Moreover, the $L^2$-error curves exhibit rapid decay (seemingly proportional to the inverse square root of the varying parameters) which aligns with Theorem\;\ref{thm:MC}\;(i).

\subsection{Example\;2} In this example we implement Algorithm \ref{alg:ours} for several multi-dimensional cases with another boundary:  Let $T=1$, $f(x)=\sin (\sum_{i=1}^dx_i)$, $(t,x)=(0,(0,\cdots,0))$, $N=100$, $M_0=2\cdot 10^6$, and $M_1=M_2 =2 \cdot 10^4$. Then for any dimension $d\in \mathbb{N}$, we randomly generate $b^o\in \mathbb{R}^d$, $\sigma^o\in \mathbb{R}^{d\times d}$ in the following way: 
\begin{itemize}[leftmargin=2em]
	\item [$\cdot$] $b^{o}:= \widetilde{b}/(\sum_{i=1}^d|\widetilde{b}^{i}|)$, where $\widetilde{b}$ is a $d$-dimensional random variable such that\footnote{We denote by $\operatorname{U}([a,b])$ for $a,b\in\mathbb{R}$ with $a<b$ the uniform distribution with values in $[a,b]$. }  $\widetilde{b}^{i}\sim \operatorname{U}([0,1])$ for every $i\in\{1,\dots,d\}$ and $\sum_{i=1}^d|\widetilde{b}^{i}|\neq 0$; 
	\item [$\cdot$] $\sigma^o:= \widetilde{\sigma}/ (\sum_{l=1}^d(\sum_{k=1}^d\widetilde{\sigma}^{k,l})^2)^{1/2}$, where $\widetilde{\sigma}=(\widetilde{\sigma}^{k,l})_{k,l\in \{1,\dots,d\}}$ is a $d\times d$-valued random variable such that $\widetilde{\sigma}^{k,l}\sim \operatorname{U}([-1,1])$  for every $k,l\in \{1,\dots,d\}$, $\widetilde{\sigma}$ is invertible, and $(\sum_{l=1}^d(\sum_{k=1}^d$ $\widetilde{\sigma}^{k,l})^2)^{1/2}\neq 0$.
\end{itemize}

Under this setup, Assumptions \ref{as:objective} and \ref{as:sigma.inverse} are satisfied. Furthermore, since $f= \sin (\sum_{i=1}^dx_i) $ is bounded and Lipschitz continuous, by Remark \ref{rem:comparison}, Assumption \ref{as:comparison} is also satisfied. Hence, the corresponding viscosity solution of the nonlinear Kolmogorov PDE is unique. We further note that unlike the semilinear form of the PDE \eqref{eq:quasi_linear} where the boundary function $f$ is convex, the corresponding Kolmogorov PDE under this setup is fully nonlinear.

\begin{table}[t]
	{\footnotesize 
		\vspace{0.1cm}
		\begin{tabular}{ccc|cccccc}
			\toprule
			\multicolumn{3}{c|}{Dimension ($=d$)} & 1 & 5 & 10 & 20 & 50 & 100 \\
			\midrule
			\multicolumn{2}{c}{\multirow{3}{*}{$v^0$ }} & {Avg. } & {\bf 0.51033} & 0.51028 & 0.51030 & 0.51062 & 0.51018 & 0.51052 \\
			& & {Std.} & {0.00038} & {0.00025} & { 0.00020} & {0.00023} & {0.00032} & {0.00026} \\
			& & $\max |\mbox{error} |$  & $0.00085$ & 0.00053 & 0.00036 & 0.00077 & 0.00075 & 0.00063 \\
			\cmidrule{1-9}
			\multirow{9}{*}{$\partial_\varepsilon v^0$}  & \multirow{3}{*}{$(1,0)$}  &  {Avg. }  & {\bf 0.45018} & 1.00558 & 1.42421 & 2.01173 & 3.18120 & 4.49794 \\
			& &   {Std.}  & {0.00267} & { 0.00451} & {0.00977} & {0.00595} & {0.00869} & {0.01729} \\
			& &   $\max |\mbox{error} |$  & $0.00449$  & 0.00909 & 0.01751 & 0.01525 & 0.02036 & 0.04313\\
			\cmidrule{4-9}
			& \multirow{3}{*}{$(0,1)$}  & { Avg. }    & {\bf 0.55718} & 1.24767 & 1.76609 & 2.49299 & 3.94644 & 5.58518 \\
			& & {Std.}    & {0.00360} & {0.00506 } & {0.00940} & {0.00822} & {0.01327} &
			{0.02033} \\
			& & $\max |\mbox{error} |$   & 0.00729  & 0.00916 & 0.02340 & 0.01613 & 0.02719 & 0.04399 \\
			\cmidrule{4-9}
			& \multirow{3}{*}{ $(1,1)$ }  & {Avg. }  & {\bf 1.00736} &  2.25325 & 3.19029 & 4.50473 &  7.12765 & 10.0831  \\
			& & {Std.}  & {0.00217} & {0.00536} & {0.00679} & {0.00806 }& {0.01324 } &  {0.02257 } \\
			& & $\max |\mbox{error} |$   & 0.00390  & 0.01124 & 0.01379 & 0.01393 & 0.02740 & 0.05303 \\
			\midrule
			\multicolumn{3}{c|}{$\lambda_{\min}(\sigma^o)$ } & 1.00 & 0.21807 & 0.06382 & 0.01078 & 0.00235 & 0.00021 \\
			\midrule
			\multicolumn{3}{c|}{Runtime in Sec. (Avg.)}  & 14.338 &  49.329 &  89.643 & 168.205 & 393.547 & 807.328 \\
			\bottomrule
		\end{tabular}
		\caption{{\small Implementation of Algorithm \ref{alg:ours} for several dimension cases. The average (Avg.) and the standard deviation ({Std.})\;of estimators for $v^0$ and $\partial_\varepsilon v^0$ (with three cases $(\gamma,\eta)=(1,0)$,\;$(0,1)$,\;$(1,1)$), and the average runtime in seconds are computed across 10 independent runs of the \textsc{Python} code. Furthermore, `$\max |\operatorname{error}|$' denotes the maximum error of the estimators for $v^0$ and $\partial_\varepsilon v^0$ over the 10 runs, where the benchmark for these errors is the average value for the scaled solution of the $d=1$ case  
				according to \eqref{eq:univ0}-\eqref{eq:univ2}
				across the 10 runs.}} \label{table:result}
		\vspace{-\baselineskip}
	}
\end{table}

To obtain a reference solution, denote by ${\Gamma}=({\Gamma}_t)_{t\in[0,T]}$ a 1-dimensional process satisfying ${\Gamma}_t= t+{W}_t^1$, $t\in[0,T]$, where ${W}^1 $ is a standard 1-dimensional Brownian motion independent of $W$. Then, by the normalization in the parameters $b^o$ and $\sigma^o$, the following property holds: for every $d\in \mathbb{N}$,
\begin{align}\label{eq:univ0}
	\mbox{law of $\sum_{i=1}^dX^{o,i}$ } \; = \; \mbox{law of ${\Gamma}$. }
\end{align}
Combined with \eqref{eq:Feyn_1} and \eqref{eq:Feyn_2} (noting that $\nabla_x f=\cos(\sum_{i=1}^dx_i)  {\bf 1}_d$ and $\mathrm{J}_x f = -\sin(\sum_{i=1}^dx_i) {\bf 1}_{d\times d}$ where here ${\bf 1}_d$ denotes the $d$-dimensional vector with value 1 in all the components and ${\bf 1}_{d\times d}$ denotes the $d\times d$-matrix with value 1 in all the components), this ensures the following characterizations: for every $d\in \mathbb{N}$,
\begin{itemize}[leftmargin=2em]
	\item [$\cdot$] $v^0(0,0) = \mathbb{E}[\sin(\sum_{i=1}^dX_T^{o,i})] = \mathbb{E}[\sin(\Gamma_T)]$;
	\item [$\cdot$] 
	$\partial_\varepsilon v^0(0,0)=\mathbb{E}[ \int_0^T   |w(t, X_t^{o})| dt]$ (when $(\gamma,\eta)=(1,0)$) is characterized by
	\begin{align}\label{eq:univ1}
		\begin{aligned}
			&\mbox{$\mathbb{E}\left[ \int_0^T    \left|\left.{\mathbb{E}}\left[\cos\left(\sum_{i=1}^d \left(X_t^{o,i}+\widetilde{X}_T^{o,i}\right)\right) \right| \widetilde{X}_t^o = 0\right]{\bf 1}_d\;\right| dt \right]$}\\
			&\quad\mbox{$=\sqrt{d} \cdot  \mathbb{E}\left[ \int_0^T   \left|{\mathbb{E}}\left[\cos({\Gamma}_t+\widetilde{\Gamma}_T) \Big| \widetilde{\Gamma}_t = 0\right]\right| dt\right]$}, 
		\end{aligned}
	\end{align}  
	with $\widetilde{X}^o$ appearing in \eqref{eq:Feyn_1}, where we denote by $\widetilde{\Gamma}_t:=t + \widetilde{W}_t^1$, $t\in[0,T]$, a 1-dimensional process with a standard 1-dimensional Brownian motion $\widetilde{W}^1$ independent of $W^1$;
	\item [$\cdot$] 
	$\partial_\varepsilon v^0(0,0)=\mathbb{E}[ \int_0^T   \|\mathrm{J}_xw(t, X_t^{o})\sigma^o\|_{\operatorname{F}} dt]$ (when $(\gamma,\eta)=(0,1)$) is given by
	\begin{align} \label{eq:univ2} 
		\begin{aligned}
			&\mbox{$\mathbb{E}\left[ \int_0^T    \left\|{\mathbb{E}}\left[\sin\left(\sum_{i=1}^d\left(X_t^{o,i}+\widetilde{X}_T^{o,i}\right)\right)\Big|\widetilde{X}_t^o = 0\right] {\bf 1}_{d\times d}\;\sigma^o \right\|_{\operatorname{F}}dt \right]$}\\ 
			&\quad\mbox{$=\sqrt{d} \cdot  \mathbb{E}\left[ \int_0^T   \left|{\mathbb{E}}\left[\sin(\Gamma_t+\widetilde{\Gamma}_T)\Big|\widetilde{\Gamma}_t =0 \right]\right| dt\right]$}. 
		\end{aligned}
	\end{align}
\end{itemize} 
\noindent
Using the observations \eqref{eq:univ0}-\eqref{eq:univ2}, we can obtain reference solutions for every $d>1$ by calculating first the $d=1$ case and then scale it with 1 (for $v^o$) and $\sqrt{d}$ (for $\partial_\varepsilon v^0)$. 

Table \ref{table:result} shows the results of several dimension cases based on 10 independent runs of a \textsc{Python} code ({\tt Code{\_}4.ipynb}) given in the link provided in Footnote\;\ref{footnote:link}. For each dimension $d$, the coefficients $b^o\in \mathbb{R}^d$ and $\sigma^o\in \mathbb{R}^{d\times d}$ are once randomly generated (as explained above) and then kept for each of the 10 runs. The values for $v^0$ obtained from applying our Algorithm \ref{alg:ours} are basically invariant and equal to $0.5103$ ($\pm 0.0003$) over the dimension $d$ whereas the value for $\partial_\varepsilon v^0$ increases proportionally to $\sqrt{d}$ for all three cases $(\gamma,\eta)\in \{(1,0),(0,1),(1,1)\}$, which is consistent with our derivation in \eqref{eq:univ0}-\eqref{eq:univ2}. Furthermore, the average runtime results show that though the complexity $\mathfrak{C}(N,M_ 0,M_ 1,M_2)$ grows quadratically in the number of samples $M_ 1$ (as in this example we choose $M_1=M_2$; see Theorem\;\ref{thm:MC}\;(ii)), the computation for high dimensional cases (e.g., $d=50,100$) is still feasible under our 
algorithm.

\begin{figure}[t]
	\centering
	\subfigure{\label{fig013}\includegraphics[scale=0.28]{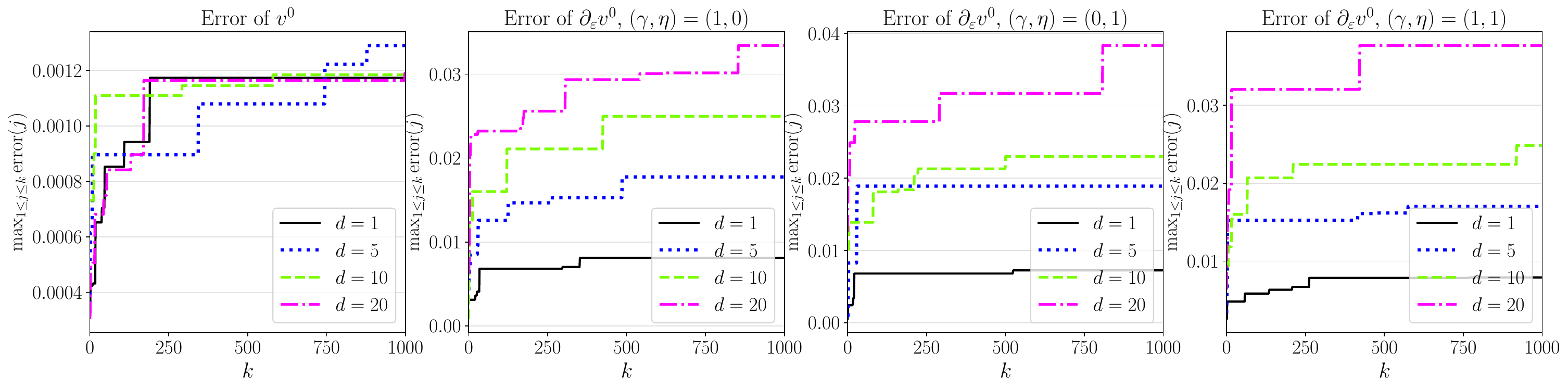}}
	\caption{
		{\small 
			The running maximum for the error between the approximation of $v^0$ and $\partial_{\varepsilon} v^0$ across $k$ runs with $k\in \{1,\dots,1000\}$, where the benchmark for these errors is the average value for the scaled solution for $d=1$ according to \eqref{eq:univ0}-\eqref{eq:univ2} across the 1,000 runs. $(b^o,\sigma^o)$ are generated randomly for every run and every~$d$.}}
	\label{fig01}
\end{figure}

%

Fig.\;\ref{fig01} shows the robustness of the estimators w.r.t.\ the input data / baseline coefficients  $b,\sigma$.
Indeed, we randomly generate $K=1000$ input parameters $(b^j,\sigma^j)_{j=1}^K$ and  plot the increasing function of $k\leq K$ that shows that maximum made error of the estimator over all $j\leq k$ input parameters.
This functions grows very modestly (seemingly logarithmic), showing robustness of the estimators.
\vspace{0.5em}

\section{Proof of Theorem \ref{thm:main}}\label{sec:thm:main}
We start by providing some notions. 
Let $t\in[0,T)$, denote by $C([t,T];\mathbb{R}^d)$  the set of all $\mathbb{R}^d$-valued continuous functions on $[t,T]$, and set 
\[
\Omega^{t}:=\{\omega=(\omega_s)_{s\in[t,T]}\in C([t,T];\mathbb{R}^d):\omega_t=0\}
\]
to be the canonical space of continuous paths. 
Let $W^t:=(W^t_s)_{s\in[t,T]}$ be the canonical process on~$\Omega^{t}$ and $\mathbb{F}^{W^t}:=({\cal F}_s^{W^t})_{s\in[t,T]}$ be the raw filtration generated by $W^t$. We equip $\Omega^{t}$ with the uniform convergence norm so that the Borel $\sigma$-field ${\cal F}^t$ on $\Omega^t$ coincides with ${\cal F}_T^{W^t}$. 
Furthermore, let $\mathbb{P}_0^t$ be the Wiener measure under which $W^t$ is a Brownian motion 
and write $\mathbb{E}^{\mathbb{P}_0^t}[\cdot]$ for the expectation under $\mathbb{P}_0^t$.

On $(\Omega^{t},{\cal F}^t,{\mathbb{F}}^{W^t},\mathbb{P}_0^t)$, consider ${X}^{t,x;o}:=({X}_s^{t,x;o})_{s\in[t,T]}$ following the baseline coefficients $b^o$ and~$\sigma^o$ and starting with $x\in\mathbb{R}^d$
, i.e.\;for $s\in[t,T]$, 
\begin{align}\label{eq:baseline_X}
	X_s^{t,x;o}=x+b^o(s-t)+\sigma^o W_s^t.
\end{align}
Moreover, let $\mathbb{L}^{t,1}(\mathbb{R}^d)$ and $\mathbb{L}^{t,1}_{\operatorname{F}}(\mathbb{R}^{d\times d})$ be the set of all ${\mathbb{F}}^{{W}^t}$-predictable processes $L$ defined on $[t,T]$ with values in $\mathbb{R}^d$ and $\mathbb{R}^{d\times d}$, respectively. 
We endow $\mathbb{L}^{t,1}(\mathbb{R}^d)$  with the norm 
$\|L\|_{\mathbb{L}^{t,1}}:=\mathbb{E}^{\mathbb{P}_0^t}[\int_t^T|L_s|ds]$ and $\mathbb{L}_{\operatorname{F}}^{t,1}(\mathbb{R}^{d\times d})$ with 
$\| L \|_{\mathbb{L}^{t,1}_{\operatorname{F}}}:=\mathbb{E}^{\mathbb{P}_0^t}[\int_t^T\|L_s\|_{\operatorname{F}}ds]$.

In analogy, we define  $\mathbb{L}^{t,\infty}(\mathbb{R}^d)$ as the set of all $\mathbb{R}^d$-valued, ${\mathbb{F}}^{{W}^t}$-predictable processes $L$ defined on $[t,T]$ that are bounded $\mathbb{P}_0^t\otimes ds$-a.e. and set $\|L \|_{\mathbb{L}^{t,\infty}}:= \inf\{C\geq 0: |L_s | \leq C\;\mbox{$\mathbb{P}_0^t\otimes ds$-a.e.}\}<\infty$.
The space  $\mathbb{L}_{\operatorname{F}}^{t,\infty}(\mathbb{R}^{d\times d})$ of $\mathbb{R}^{d \times d}$-valued processes is defined analogously to $\mathbb{L}^{t,\infty}(\mathbb{R}^d)$, with $|\cdot|$ replaced  by $\|\cdot\|_{\rm F}$ in the definition of~$\| L \|_{\mathbb{L}^{t,\infty}_{\rm F}}$.

For any $(b,\sigma)\in \mathbb{L}^{t,\infty}(\mathbb{R}^d)\times \mathbb{L}^{t,\infty}_{\operatorname{F}}(\mathbb{R}^{d\times d})$, we define an It\^o $({\mathbb{F}}^{{W}^t},\mathbb{P}^t_0)$-semimartingal $X^{t,x;b,\sigma}=(X^{t,x;b,\sigma}_s)_{s\in[t,T]}$ starting with $x\in\mathbb{R}^d$ by 
\begin{align}\label{eq:ito.semi.time}
	X_s^{t,x;b,\sigma} :=x+  \int_t^s b_u du  +\int_t^s \sigma_u d W^t_u,\quad s\in[t,T],
\end{align}
and note that $X_s^{t,x; o }=X_s^{t,x; b^o,\sigma^o }$ $s\in[t,T]$; see \eqref{eq:baseline_X}.
Moreover, for any $\varepsilon\geq 0$ and $t\in[0,T)$, denote by
\begin{align}\label{eq:char.strong}
	\mathcal{C}^{\varepsilon}(t)
	:=\left\{(b,\sigma)\in \mathbb{L}^{t,\infty}(\mathbb{R}^d)\times \mathbb{L}^{t,\infty}_{\operatorname{F}}(\mathbb{R}^{d\times d})~|~(b_s,\sigma_s)\in \mathcal{B}^\varepsilon\;\;\mbox{$\mathbb{P}_0^t\otimes ds$-a.e.}\right\}
\end{align}
the set of all ${\mathbb{F}}^{{W}^t}$-predictable processes taking values within the $\varepsilon$-neighborhood $\mathcal{B}^\varepsilon$ of the baseline coefficients $(b^o,\sigma^o)$ given in \eqref{eq:uncertain.char}.

Let us start by  providing some a priori estimates for  $X_T^{t,x;b,\sigma}$.

\begin{lem}\label{lem:priori_estimate} 
	For every $p\geq 1$, there is a constant $C_p>0$ s.t.~the following~holds:
	\begin{itemize}[leftmargin=2em]
		\item [(i)] For all $(t,x)\in[0,T)\times\mathbb{R}^d$, we have that $\mathbb{E}^{\mathbb{P}_0^t}[  | X_T^{t,x;o} |^p  ] \leq C_p(1+|x|^p)$.
		\item[(ii)] For all $\varepsilon \geq 0$ and $(t,x)\in [0,T)\times \mathbb{R}^d$, 
		$
		\sup_{(b,\sigma)\in {\cal C}^\varepsilon(t)}  \mathbb{E}^{\mathbb{P}_0^t}[  | X_T^{t,x;b,\sigma}-X_T^{t,x;o} |^p  ]
		\leq C_p   \varepsilon^p$.
	\end{itemize}
	
\end{lem}
\begin{proof} 
	We only prove (ii), as the proof for (i) follows the same line of reasoning.

	Fix $\varepsilon\geq 0$ and $(t,x)\in [0,T)\times \mathbb{R}^d$, and let $(b,\sigma)\in {\cal C}^\varepsilon(t)$.
	Then
	\[X_T^{t,x;b,\sigma}-X_T^{t,x;o}
	= \int_t^T(b_s-b^o)ds + \int_t^T (\sigma_s-\sigma^o) dW^t_s.\]
	We estimate both terms separately.
	By Jensen's inequality and the definition of $ {\cal C}^\varepsilon(t)$,
	\[ \mathbb{E}^{\mathbb{P}_0^t}\bigg[ \Big| \int_t^T(b_s-b^o)ds \Big|^p \bigg]
	\leq \mathbb{E}^{\mathbb{P}_0^t}\bigg[ (T-t)^{p-1} \int_t^T  |b_s-b^o|^p ds \bigg]
	\leq (T-t)^{p} \varepsilon^p.\]
	Moreover, if $c_{{\rm BDG},p}>0$ denotes the constant appearing in the Burkholder-Davis-Gundy (BDG) inequality (see, e.g, \cite[Theorem 92, Chap. VII]{dellacherieprobabilities}), then 
	\begin{align*}
		\mathbb{E}^{\mathbb{P}_0^t}\bigg[ \Big| \int_t^T (\sigma_s-\sigma^o )dW^t_s\Big|^p \bigg] 
		&\leq c_{{\rm BDG},p} \mathbb{E}^{\mathbb{P}_0^t}\bigg[ \Big( \int_t^T\| \sigma_s-\sigma^o\|_{\operatorname{F}}^2ds  \Big)^{p/2} \bigg] 
		\\
		&\leq  c_{{\rm BDG},p} (T-t)^{p/2}  \varepsilon^p, 
	\end{align*}
	where the second inequality follows from the definition of $ {\cal C}^\varepsilon(t)$.
	Thus the proof is completed using the elementary inequality $(a+b)^p\leq 2^p (a^p+b^p)$ for all $a,b\geq 0$.
\end{proof}

\begin{rem}\label{rem:f.nablaf.integrable}	
	By Assumption \ref{as:objective}, the (weak) Hessian $D_x^2 f$ has at most polynomial growth of order $\alpha$.
	In particular, there is a constant $\widetilde{C}_f>0$ (that depends on $C_f$ in Assumption~\ref{as:objective}) such that for every $(x,y)\in \mathbb{R}^d$, 
	\[
	\left|f(y)-f(x)-\nabla_x^\top f(x) (x-y) \right|
	\leq \widetilde{C}_f\left(1 +|x|^\alpha+|y|^\alpha\right) \cdot |y-x|^2.
	\] 
	Moreover, $\nabla_x f$  and $f$ have at most polynomial growth of of order $\alpha+1$  and $\alpha+2$, respectively.
	
	Next note that if $g\colon\mathbb{R}^d\to\mathbb{R}$ is any function with at most polynomial growth, then Lemma \ref{lem:priori_estimate} implies that $g(X^{t,x;b,\sigma}_T)$ is integrable for every $(b,\sigma)\in\mathcal{C}^\varepsilon(t)$.
	Therefore $f(X^{t,x;b,\sigma}_T)$, $\partial_{x_i}f(X^{t,x;b,\sigma}_T)$,  $|\partial_{x_i}f(X^{t,x;b,\sigma}_T)|^2$, \ldots are integrable.
\end{rem}

\begin{lem}\label{lem:BSDE}
	Suppose that Assumptions \ref{as:objective} and \ref{as:sigma.inverse} are satisfied. For $i=1,\dots,d$, let $w^i:[t,T]\times\mathbb{R}^d\rightarrow \mathbb{R}$ be the unique classical solution of \eqref{eq:linear.pde.1stdev} with polynomial growth (see Proposition~\ref{pro:main}\;(ii)). Let $w: [t,T]\times \mathbb{R}^d\rightarrow \mathbb{R}^d$ 
	and $\mathrm{J}_x w: [t,T]\times \mathbb{R}^d\rightarrow \mathbb{R}^{d\times d}$ be given in \eqref{eq:vector.value.jacobi}, let $(t,x)\in[0,T]\times \mathbb{R}^d$, and set
	\begin{align}
		\label{eq:def.Y.Z}
		\qquad\quad  Y_s^{t,x} := w\left(s, X_s^{t,x;o}\right),\qquad
		Z_s^{t,x} := \mathrm{J}_x w\left(s, X_s^{t,x;o}\right) \sigma^o, \quad s\in[t,T].
	\end{align}
	Then, for every $i=1,\dots,d$, 
	\[
	Y^{t,x,i}_s = \partial_{x_i} f (X_T^{t,x;o})  -\int_s^T (Z^{t,x,i}_r)^\top dW^t_r,\qquad s\in[t,T],\]
	with $Y^{t,x,i}$ and $(Z^{t,x,i})^\top$ denoting the $i$-th component of $Y^{t,x}$ and the $i$-th row vector of~$Z^{t,x}$, respectively. In particular, $Y_s^{t,x}= \mathbb{E}^{\mathbb{P}_0^t}[\nabla_x f(X_T^{t,x,o})|{\cal F}_s^{W^t}]$ for every $s\in[t,T]$.
\end{lem}
\begin{proof}[Proof of Lemma \ref{lem:BSDE}]
	Fix $i\in\{1,\dots,d\}$.
	Since $w^i\in C^{1,2}([t,T)\times \mathbb{R}^d)$ (see Proposition \ref{pro:main}\;(ii) and Section~\ref{sec:proof:pro:main}), an application of It\^o's formula ensures that for~$s\in[t,T]$,
	\begin{align*}
		&w^i(T,X_T^{t,x;o}) - w^i(s,X_s^{t,x;o})
		=\int_s^T \nabla_x^\top w^i(r,X_r^{t,x;o})\;\sigma^o dW_r^{t} \\
		&+\int_s^T \bigg(\partial_rw^i(r,X_r^{t,x;o})+  \langle b^o, \nabla_x w^i (r,X_r^{t,x;o}) \rangle +  \frac{1}{2} \operatorname{tr} \left((\sigma^o) (\sigma^o)^\top D^2_{x}w^i(r,X_r^{t,x;o})\right) \bigg) dr.
	\end{align*}
	The  second integral is equal to zero because $w^i$ solves the linear Kolmogorov PDE given in~\eqref{eq:linear.pde.1stdev}.
	Therefore, using the boundary condition $w^i(T,\cdot) =\partial_{x_i}f(\cdot)$ and the definitions of $Y^{t,x}$ and $Z^{t,x}$ given in \eqref{eq:def.Y.Z}, we conclude that  for every $s\in[t,T]$,
	\begin{align}
		\label{eq:take.cond.exp}
		Y^{t,x,i}_s= \partial_{x_i}f(X_T^{t,x;o}) -\int_s^T (Z^{t,x,i}_u)^\top dW_u^{t},
	\end{align}
	as claimed.
	
	The `in particular' part follows by taking conditional expectations in \eqref{eq:take.cond.exp}.
	Indeed,  by Remark \ref{rem:f.nablaf.integrable}, $\partial_{x_i}f(X_T^{t,x;o})$ and $Y_t^{t,x}= w\left(t, x\right)$ are square integrable (because $\partial_{x_i}f$ and $w$ have polynomial growth).
	Therefore, it follows from {\eqref{eq:take.cond.exp}} that $\int_\cdot^T (Z^{t,x,i}_s)^\top dW_s^{t}$ is a square integrable martingale. 
	
	Furthermore, since $Y_s^{t,x,i}=w^i(s,X_s^{t,x;o})$ is ${\cal F}_s^{W^t}$-measurable, 
\begin{align*}
	Y^{t,x,i}_s 
	= \mathbb{E}^{\mathbb{P}_0^t} \left[ Y^{t,x,i}_s  \middle| {\cal F}_s^{W^t} \right]
	&= \mathbb{E}^{\mathbb{P}_0^t}\left[ \partial_{x_i}f(X_T^{t,x;o}) - \int_s^T (Z^{t,x,i}_r)^\top dW_r^{t} \middle|{\cal F}_s^{W^t} \right]\\
	&=\mathbb{E}^{\mathbb{P}_0^t}\left[ \partial_{x_i}f(X_T^{t,x;o}) \middle|{\cal F}_s^{W^t} \right],
\end{align*}
as claimed.
\end{proof}

For sufficiently integrable $\mathbb{R}^d$-valued processes $(L_s)_{s\in[t,T]}$\;and\;$(M_s)_{s\in[t,T]}$, define  $\langle L,M \rangle_{\mathbb{P}_0^t\otimes ds}$ $:= \mathbb{E}^{\mathbb{P}_0^t}[\int_t^T \langle L_s, M_s \rangle ds ]$.
In a similar manner, we set $\langle L,M \rangle_{\mathbb{P}_0^t\otimes ds, \operatorname{F}}:= \mathbb{E}^{\mathbb{P}_0^t}[\int_t^T \langle L_s, M_s \rangle_{\operatorname{F}} ds]$ for $\mathbb{R}^{d\times d}$-valued processes.

\begin{lem}\label{lem:inner} 
Suppose that Assumptions \ref{as:objective} and \ref{as:sigma.inverse} are satisfied and let $Y^{t,x}$,$Z^{t,x}$ be the processes defined in \eqref{eq:def.Y.Z}.
Then, for every  $\varepsilon \geq 0$ and $(b,\sigma)\in {\cal C}^\varepsilon(t)$, we have~that
\begin{align*}
	\mathbb{E}^{\mathbb{P}_0^t}\left[\nabla^\top_x f(X_T^{t,x;o}) \left(X_T^{t,x;b,\sigma}-X_T^{t,x;o}\right)\right] = \langle Y^{t,x}, b-b^o \rangle_{\mathbb{P}_0^t\otimes ds} + \langle Z^{t,x}, \sigma-\sigma^o \rangle_{\mathbb{P}_0^t\otimes ds, \operatorname{F}}.
\end{align*} 
\end{lem}
\begin{proof}
For $i=1,\dots,d$, denote by  $(b_s^{i}-b^{o,i})_{s\in[t,T]}$ and $(\sigma_s^{i}-\sigma^{o,i})_{s\in[t,T]}$  the $i$-th component  of $b-b^o$ and $i$-th row vector  of $\sigma-\sigma^o$, respectively.
Using this notation,
\begin{align*}
	&\nabla^\top_x f(X_T^{t,x;o}) \left(X_T^{t,x;b,\sigma}-X_T^{t,x;o}\right) =\sum_{i=1}^d \left( \Xi^{b,i} +  \Xi^{\sigma,i} \right)\\
	&\qquad :=\sum_{i=1}^d \left( \partial_{x_i} f(X_T^{t,x;o}) \int_t^T(b_s^{i}-b^{o,i})ds  + \partial_{x_i} f(X_T^{t,x;o})\int_t^T(\sigma_s^{i}-\sigma^{o,i})dW_s^{t}\right).
\end{align*}
It follows from Remark \ref{rem:f.nablaf.integrable} that $ \Xi^{b,i},\Xi^{\sigma,i}$ are integrable (noting that $b-b^0$ and $\sigma-\sigma^0$ are bounded uniformly).
In particular, 
\[\mathbb{E}^{\mathbb{P}_0^t}\left[\nabla^\top_x f(X_T^{t,x;o}) (X_T^{t,x;b,\sigma}-X_T^{t,x;o}) \right ]
=\sum_{i=1}^d \left( \mathbb{E}^{\mathbb{P}_0^t}\left[\Xi^{b,i} \right]  +  \mathbb{E}^{\mathbb{P}_0^t}\left[\Xi^{\sigma,i} \right]  \right)\]
and  it remains to show that 
\[\mathbb{E}^{\mathbb{P}_0^t}[\Xi^{b,i}] = \langle Y^{t,x,i}, b^i-b^{o,i}\rangle_{\mathbb{P}_0^t\otimes ds}
\quad\text{and}\quad
\mathbb{E}^{\mathbb{P}_0^t}[\Xi^{\sigma,i}] = \langle Z^{t,x,i}, (\sigma^i-\sigma^{o,i})^\top\rangle_{\mathbb{P}_0^t\otimes ds}\]
for every $i=1,\dots,d$.
To that end, fix such $i$.

\vspace{0.5em}
We first claim that  $\mathbb{E}^{\mathbb{P}_0^t}[\Xi^{b,i}] = \langle Y^{t,x,i}, b^i-b^{o,i}\rangle_{\mathbb{P}_0^t\otimes ds}$.
Indeed, an application of Fubini's theorem shows that
\begin{align*}
	\mathbb{E}^{\mathbb{P}_0^t}\left[\Xi^{b,i}\right] 
	&= \int_t^T \mathbb{E}^{\mathbb{P}_0^t}\left[\mathbb{E}^{\mathbb{P}_0^t}\left[\partial_{x_i} f(X_T^{t,x;o})| {\cal F}_s^{W^t}\right](b_s^i-b^{o,i})\right]ds \\
	&=\mathbb{E}^{\mathbb{P}_0^t}\left[\int_t^T Y_s^{t,x,i}(b^i_s-b^{o,i})ds\right]
	= \langle Y^{t,x,i}, b^i-b^{o,i}\rangle_{\mathbb{P}_0^t\otimes ds},
\end{align*}
where  the second equality holds because $Y_s^{t,x,i}=\mathbb{E}^{\mathbb{P}_0^t}[\partial_{x_i} f(X_T^{t,x;o})| {\cal F}_s^{W^t}]$, $s\in[t,T]$; see Lemma~\ref{lem:BSDE}.

\vspace{0.5em}
Next, we claim that $ \mathbb{E}^{\mathbb{P}_0^t}[\Xi^{\sigma,i}]=\langle Z^{t,x,i}, (\sigma^i-\sigma^{o,i})^\top \rangle_{\mathbb{P}_0^t\otimes ds}.$
Indeed, by Lemma \ref{lem:BSDE}, 
\begin{align*}
		\mathbb{E}^{\mathbb{P}_0^t}\left[\Xi^{\sigma,i}\right] 
		&= \mathbb{E}^{\mathbb{P}_0^t}\left[\left( \int_t^T(Z_s^{t,x,i})^\top dW^t_s +Y_t^{t,x,i} \right)\int_t^T(\sigma_s^{i}-\sigma^{o,i})dW_s^{t}\right] .
\end{align*}
Further, by the It\^o-isometry,
\begin{align*}
	\mathbb{E}^{\mathbb{P}_0^t}\left[\int_t^T(Z_s^{t,x,i})^\top dW_s^t\int_t^T(\sigma_s^{i}-\sigma^{o,i})dW_s^{t}\right] &=\mathbb{E}^{\mathbb{P}_0^t}\left[\int_t^T(Z_s^{t,x,i})^\top(\sigma_s^{i}-\sigma^{o,i})^\top ds \right] \\
	&=\langle  Z^{t,x,i}, (\sigma^{i}-\sigma^{o,i})^\top \rangle_{\mathbb{P}_0^t\otimes ds},
\end{align*}
and since $Y_t^{t,x,i}$ is ${\cal F}_t^{W^t}$-measurable (see \eqref{eq:def.Y.Z} given in Lemma \ref{lem:BSDE}), $
\mathbb{E}^{\mathbb{P}_0^t}[ Y_t^{t,x,i}\int_t^T(\sigma_s^{i}-\sigma^{o,i})dW_s^{t} ] = 0$ by the martingale property of $W^t$.
This completes the proof.
\end{proof}

In Section \ref{sec:StrongWeak}, we shall show that if Assumptions \ref{as:objective}, \ref{as:sigma.inverse}, and \ref{as:comparison} are satisfied, then  the unique viscosity solution $v^\varepsilon$ of \eqref{eq:nonlinear.pde} satisfies the following:
For all $\varepsilon< \lambda_{\min}(\sigma^o)$  and $(t,x)\in [0,T)\times \mathbb{R}^d$, we have that 
\begin{align}\label{eq:strong_preview}
v^\varepsilon (t,x) = \sup_{(b,\sigma)\in \mathcal{C}^{\varepsilon}(t)}\mathbb{E}^{\mathbb{P}_0^t}\left[f\left(X^{t,x;b,\sigma}_T\right)\right],
\end{align}
with $v^\varepsilon(T,\cdot)=f(\cdot)$, see Lemma \ref{lem:weak_viscosity} and Proposition \ref{pro:equiv.weak.strong}. 
The formula for $v^\varepsilon$ given in \eqref{eq:strong_preview} will be crucial in the following proof.

\begin{lem}\label{lem:main_step1} 
Suppose that Assumptions \ref{as:objective}, \ref{as:sigma.inverse}, and \ref{as:comparison} are satisfied and, for every $(t,x)\in[0,T)\times \mathbb{R}^d$,  let $Y^{t,x}, Z^{t,x}$ be the processes defined in \eqref{eq:def.Y.Z}.
Moreover, let $\alpha$ be as in Assumption \ref{as:objective}.
Then, there exists a constant $c$ independent of $t,x,\varepsilon$ such that for every $\varepsilon<\min\{1,\lambda_{\min}(\sigma^0)\}$ and $(t,x)\in[0,T)\times \mathbb{R}^d$, we have that
\begin{align*}
	&\left| v^\varepsilon(t,x) - \left( v^0(t,x) + \sup_{(b,\sigma)\in {\cal C}^\varepsilon(t)}\left(\langle Y^{t,x}, b-b^o\rangle_{\mathbb{P}_0^t\otimes ds} + \langle Z^{t,x}, \sigma-\sigma^o\rangle_{\mathbb{P}_0^t\otimes ds, \operatorname{F}} \right) \right)\right| \\
	&\leq c (1+|x|^\alpha) \varepsilon^2.
\end{align*}
\end{lem}
\begin{proof}
Fix $\varepsilon$ as in the lemma and recall the  formula for $v^\varepsilon$ given in~\eqref{eq:strong_preview};
in particular $v^0(t,x) = \mathbb{E}^{\mathbb{P}_0^t}[f(X^{t,x;o}_T)].$

Next, using Remark \ref{rem:f.nablaf.integrable},  for any $(b,\sigma)\in \mathcal{C}^\varepsilon(t)$,
\begin{align*}	
	&\left|f\big(X_T^{t,x;b,\sigma}\big)- f\big(X_T^{t,x;o}\big) -\nabla_x^\top f\big(X_T^{t,x;o}\big)\left(X_T^{t,x;b,\sigma}-X_T^{t,x;o}\right)\right| \\
	&\quad \leq  \widetilde{C}_f \cdot \left( 1+ \big|X_T^{t,x;b,\sigma} \big|^\alpha + \left|X_T^{t,x;o} \right|^\alpha \right)\cdot \left| X_T^{t,x;b,\sigma}-X_T^{t,x;o}\right|^2
	=: {\rm I}^{b,\sigma}.
\end{align*}

We claim that there is $c>0$ that depends only on $\alpha,\widetilde{C}_f$ (see Remark \ref{rem:f.nablaf.integrable})
such~that
$\sup_{(b,\sigma)\in\mathcal{C}^\varepsilon(t)}\mathbb{E}^{\mathbb{P}_0^t}[{\rm I}^{b,\sigma} ] 
\leq c (1 + |x|^{\alpha}) \varepsilon^2.$
To that end, an application of the Cauchy-Schwartz inequality together with the elementary inequality $(1+a+b)^2 \leq 3^2(1+a^2+b^2)$ for all $a,b\geq0$ shows that 		
\begin{align*}
	\mathbb{E}^{\mathbb{P}_0^t}\left[ {\rm I}^{b,\sigma} \right]
	&\leq \widetilde{C}_f 3 \cdot \mathbb{E}^{\mathbb{P}_0^t}\left[  1+ \big|X_T^{t,x;b,\sigma} \big|^{2\alpha} + \left|X_T^{t,x;o} \right|^{2\alpha} \right]^{1/2}
	\mathbb{E}^{\mathbb{P}_0^t}\Big[\big| X_T^{t,x;b,\sigma}-X_T^{t,x;o}\big|^4 \Big]^{1/2}.
\end{align*}
Moreover, we have by Lemma \ref{lem:priori_estimate} that
\[\sup_{(b,\sigma)\in\mathcal{C}^\varepsilon(t)}\mathbb{E}^{\mathbb{P}_0^t}\Big[\big| X_T^{t,x;b,\sigma}-X_T^{t,x;o}\big|^4 \Big]^{1/2} 
\leq C_4^{1/2}\varepsilon^2,\]
where $C_4$ is the constant appearing in Lemma \ref{lem:priori_estimate}.
Furthermore, as $\varepsilon<1$, another application of Lemma \ref{lem:priori_estimate} together with the inequality $(a+b)^{2\alpha}\leq 2^{2\alpha}(a^{2\alpha}+b^{2\alpha})$ for all $a,b\geq 0$ implies~that 
\begin{align*}
	\begin{aligned}
		&\sup_{(b,\sigma)\in\mathcal{C}^\varepsilon(t)} \mathbb{E}^{\mathbb{P}_0^t}\left[  1+ \big|X_T^{t,x;b,\sigma} \big|^{2\alpha} + \left|X_T^{t,x;o} \right|^{2\alpha} \right]^{\frac{1}{2}}\\
		&\quad \leq \sup_{(b,\sigma)\in\mathcal{C}^\varepsilon(t)} \left( 1+2^{2\alpha} \mathbb{E}^{\mathbb{P}_0^t}\left[\big|X_T^{t,x;b,\sigma}- X_T^{t,x;o} \big|^{2\alpha}\right] + (2^{2\alpha}+1 )\mathbb{E}^{\mathbb{P}_0^t}\left[\big| X_T^{t,x;o} \big|^{2\alpha}\right] \right)^{\frac{1}{2}}\\
		&\quad \leq  \Big( 1+2^{2\alpha} C_{2\alpha} + (2^{2\alpha}+1)C_{2\alpha}(1+|x|^{2\alpha})  \Big)^{\frac{1}{2}} \leq \left(1+2^{2\alpha+1}C_{2\alpha} +C_{2\alpha}\right)^{\frac{1}{2}} (1+|x|^{\alpha}),
	\end{aligned}
\end{align*}
where $C_{2\alpha}$ is the constant appearing in Lemma \ref{lem:priori_estimate}. Our claim follows by setting $c:=\widetilde{C}_f 3 C_4^{1/2} (1+2^{2\alpha+1}C_{2\alpha} +C_{2\alpha})^{{1}/{2}}$.  

Finally, combining all the previous estimates we conclude that
\begin{align*}
	&\left| v^\varepsilon(t,x) - \left( v^0(t,x)
	+\sup_{(b,\sigma)\in {\cal C}^\varepsilon(t)} \mathbb{E}^{\mathbb{P}_0^t}\left[ \nabla_x^\top f\big(X_T^{t,x;o}\big)\left(X_T^{t,x;b,\sigma}-X_T^{t,x;o}\right)\right] \right)\right|\\
	&\leq c(1+|x|^\alpha) \varepsilon^2.
\end{align*}
Thus, the proof is completed by using \eqref{eq:strong_preview} and applying Lemma~\ref{lem:inner}.
\end{proof}

\begin{lem}\label{lem:main_step2} 
Suppose that Assumptions \ref{as:objective} and \ref{as:sigma.inverse} are satisfied.
Then, for every fixed $(t,x)\in[0,T)\times \mathbb{R}^d$ and $\varepsilon\geq 0$,
\begin{align*}
	\begin{split}
		\Phi:=&\sup_{(b,\sigma)\in {\cal C}^\varepsilon(t)}\left(\langle Y^{t,x}, b-b^o\rangle_{\mathbb{P}_0^t\otimes ds} + \langle Z^{t,x}, \sigma - \sigma^o\rangle_{\mathbb{P}_0^t\otimes ds,\operatorname{F}}\right)\\
		&= \varepsilon\cdot \left(\gamma \|Y^{t,x} \|_{\mathbb{L}^{t,1}}+\eta \|Z^{t,x} \|_{\mathbb{L}_{\operatorname{F}}^{t,1}}\right) 
		=:\Psi .
	\end{split}
\end{align*}
\end{lem}

    The proof of Lemma \ref{lem:main_step2} follows from the norm duality between $\mathbb{L}^{t,\infty}$ and $\mathbb{L}^{t,1}$, together with the definition of $\mathcal{C}^\varepsilon(t)$.
    We present it for completeness.

\begin{proof}
Fix $\varepsilon$.
We first claim that $\Phi\leq \Psi$.
To that end, set $\mathcal{C}^\varepsilon_1(t):= \{b:(b,\sigma)\in {\cal C}^\varepsilon(t)\}$, $\mathcal{C}^\varepsilon_2(t):= \{\sigma :(b,\sigma)\in {\cal C}^\varepsilon(t)\}$ so that $\mathcal{C}^\varepsilon(t)=\mathcal{C}^\varepsilon_1(t)\times \mathcal{C}^\varepsilon_2(t)$.
Using the Cauchy-Schwartz inequality in $\mathbb{R}^d$ and  H\"older's inequality (with exponents 1 and $\infty$),
\begin{align*}
	\begin{aligned}
		\sup_{b\in {\cal C}^\varepsilon_1(t)}\langle Y^{t,x}, b-b^o\rangle_{\mathbb{P}_0^t\otimes ds}&\leq \sup_{b\in {\cal C}^\varepsilon_1(t)}\mathbb{E}^{\mathbb{P}_0^t}\left[\int_t^T|Y_s^{t,x}| |b_s-b^o|ds \right]
		\leq  \|Y^{t,x}\|_{\mathbb{L}^{t,1}}\;\varepsilon \gamma .
	\end{aligned}
\end{align*}
In a similarly manner,
\begin{align*}
	\begin{aligned}
		\sup_{\sigma\in {\cal C}^\varepsilon_2(t)} \langle Z^{t,x}, \sigma - \sigma^o\rangle_{\mathbb{P}_0^t\otimes ds,\operatorname{F}}\leq \sup_{\sigma\in {\cal C}^\varepsilon_2(t)}\mathbb{E}^{\mathbb{P}_0^t}\left[\int_t^T\|Z_s^{t,x}\|_{\operatorname{F}} \|\sigma_s-\sigma^o\|_{\operatorname{F}}ds \right]
		\leq \|Z^{t,x}\|_{\mathbb{L}_{\operatorname{F}}^{t,1}} \;\varepsilon \eta .
	\end{aligned}
\end{align*}
The combination of these two estimates shows that indeed  $\Phi\leq \Psi$.

\vspace{0.5em}
Next we claim that  $\Phi\geq \Psi$.
Define $\tilde{\sigma}^*\in\mathbb{L}_{\operatorname{F}}^{t,\infty}$ by
\begin{align*}
	\tilde{\sigma}^*_s
	:=  \left\{
	\begin{aligned}
		&Z_s^{t,x} / \| Z_s^{t,x}\|_{\operatorname{F}}\;\; &&\mbox{if}\quad\|Z^{t,x}_s\|_{\operatorname{F}}> 0; 
		\\
		&\;0\quad &&\mbox{else},
	\end{aligned}
	\right.
\end{align*}
which satisfies $\|\tilde{\sigma}^*\|_{\mathbb{L}_{\operatorname{F}}^{t,\infty}}\leq 1$ and  $
\langle Z_s^{t,x} , \tilde{\sigma}_s^* \rangle_{\operatorname{F}}
= \|Z_s^{t,x}\|_{\operatorname{F}}$. 
This implies that
\begin{equation}\label{eq:sig_ineq_op}
	\|Z^{t,x}\|_{\mathbb{L}_{\operatorname{F}}^{t,1}}
	=\mathbb{E}^{\mathbb{P}_0^t}\left[\int_t^T \langle Z_s^{t,x}, \tilde{\sigma}_s^*\rangle_{\operatorname{F}}ds \right] =\langle Z^{t,x}, \tilde{\sigma}^* \rangle_{\mathbb{P}_0^t\otimes ds,\operatorname{F}}. 
\end{equation}
In a similar manner,  we can construct some $\tilde{b}^*\in \mathbb{L}^{t,\infty}(\mathbb{R}^d)$ satisfying $\| \tilde{b}^* \|_{\mathbb{L}^{t,\infty}}\leq 1$~and
\begin{align}\label{eq:b_ineq_op}
	\|Y^{t,x} \|_{\mathbb{L}^{t,1}} 
	=\mathbb{E}^{\mathbb{P}_0^t}\left[\int_t^T \langle Y_s^{t,x}, \tilde{b}_s^*\rangle ds \right] = \langle Y^{t,x}, \tilde{b}^* \rangle_{\mathbb{P}_0^t\otimes ds}.
\end{align}

Now define $(b^{\ast},\sigma^{\ast})
:= (b^o+\varepsilon \gamma \tilde{b}^*,\sigma^o+\varepsilon \eta \tilde{\sigma}^* ) \in {\cal C}^\varepsilon(t).$
Then, by  \eqref{eq:b_ineq_op} and \eqref{eq:sig_ineq_op},
\begin{align*}
		\Phi&\geq \langle Y^{t,x}, b^{\ast}-b^o\rangle_{\mathbb{P}_0^t\otimes  ds} + \langle Z^{t,x}, \sigma^{\ast} - \sigma^o\rangle_{\mathbb{P}_0\otimes_t ds,\operatorname{F}}\\
		&=  \varepsilon \cdot \left( \gamma \langle Y^{t,x}, \tilde{b}^{*}\rangle_{\mathbb{P}_0^t\otimes  ds} +\eta \langle Z^{t,x}, \tilde{\sigma}^* \rangle_{\mathbb{P}_0^t\otimes ds,\operatorname{F}} \right)\\
		&= \varepsilon \cdot \left(\gamma \|Y^{t,x} \|_{\mathbb{L}^{t,1}}+\eta \|Z^{t,x} \|_{\mathbb{L}_{\operatorname{F}}^{t,1}}\right) = \Psi.  \qedhere
\end{align*}
\end{proof}

\begin{proof}[Proof of Theorem \ref{thm:main}]
Fix  $(t,x)\in[0,T)\times \mathbb{R}^d$ and let $(Y^{t,x},Z^{t,x})$ be the processes defined in~\eqref{eq:def.Y.Z}, that is, 
$Y_s^{t,x} = w(s, X_s^{t,x;o})$ and $
Z_s^{t,x} = \mathrm{J}_xw(s, X_s^{t,x;o}) \sigma^o$ for $s\in[t,T]$.
Then, by Lemmas \ref{lem:main_step1} and \ref{lem:main_step2}, 
we have for every $\varepsilon<\min\{1,\lambda_{\min}(\sigma^0)\}$ that  
\begin{align*}
\left| v^\varepsilon(t,x)
-\Big(v^0(t,x) +  \varepsilon\cdot  \big(\gamma \|Y^{t,x} \|_{\mathbb{L}^{t,1}}+\eta \|Z^{t,x} \|_{\mathbb{L}_{\operatorname{F}}^{t,1}}\big) \Big) \right| \leq  c (1+|x|^\alpha) \varepsilon^2,
\end{align*} 
where $c>0$ is the constant (that is independent of $t,x,\varepsilon$) appearing in Lemma \ref{lem:main_step1}. The proof follows from the definitions of the norms on $\mathbb{L}^{t,1}(\mathbb{R}^d)$ and $\mathbb{L}^{t,1}_{\rm F}(\mathbb{R}^{d\times d})$, and since the law of $(X^{t,x,o}_s)_{s\in[t,T]}$ under $\mathbb{P}^t_0$ is equal to the conditional law of $(x+X^o_s)_{s\in[t,T]}$ under $\mathbb{P}$ given $X_t^o=0$, where $(X^o_t)_{t\in[0,T]}$ is the process defined in Theorem \ref{thm:main}.
\end{proof}

\section{Weak and strong formulation of nonlinear Kolmogorov PDE}\label{sec:Weak}
\subsection{Semimartingale measures}\label{sec:prelimi_weak}
In this section we adopt a framework for semimartingale uncertainty introduced by \cite{NeufeldNutz2014,neufeld2017nonlinear}. For any $(t,x)\in [0,T)\times \mathbb{R}^d$, denote~by 
\[
\Omega^{t,x}:=\left\{\omega=(\omega_s)_{s\in[t,T]}\in C([t,T];\mathbb{R}^d):\omega_t=x\right\}
\]
under which $X^{t}:=(X^t_s)_{s\in[t,T]}$ is the corresponding canonical process starting in $x$. Furthermore, let $\mathbb{F}^{X^t}:=({\cal F}_s^{X^t})_{s\in[t,T]}$ be the raw filtration generated by~$X^t$. 
{We equip $\Omega^{t,x}$ with the uniform norm $\| \omega \|_{t,\infty}:=\max_{t\leq s\leq T} | \omega_s |$ so that the Borel $\sigma$-field on $\Omega^{t,x}$ coincides with ${\cal F}_T^{X^t}$.}  

For $t=0$, we set for shorthand notation $X:=X^0$ and $\| \omega \|_{\infty}:=\| \omega \|_{0,\infty}$.
Denote by ${\cal P}( \Omega^{0,x})$ the set of all Borel probability measures on $ \Omega^{0,x}$. For each $p\in \mathbb{N}$, set 
\begin{align}\label{eq:p.msr.set}
{\cal P}^p( \Omega^{0,x}):= \left\{\mathbb{P}\in{\cal P}( \Omega^{0,x})~\Big|~\int_{\Omega^{0,x}} \|\omega\|_{\infty}^p \mathbb{P}(d \omega)<\infty \right\}
\end{align}
to be the subset of all Borel probability measures on $\Omega^{0,x}$ with finite $p$-th moment. Furthermore, let $C(\Omega^{0,x};\mathbb{R})$ be the set of all continuous functions from $\Omega^{0,x}$ to $\mathbb{R}$ and~set
\begin{align}\label{eq:p.mmt.set}
C_p(\Omega^{0,x};\mathbb{R}):= \left\{\xi\in C(\Omega^{0,x};\mathbb{R})~\Big|~\|\xi\|_{C_p}:=\sup_{\omega\in \Omega^{0,x}} \frac{|\xi(\omega)|}{1+\|\omega\|_{\infty}^p}<\infty\right\}.
\end{align}
We equip ${\cal P}^p( \Omega^{0,x})$  with the  topology $\tau_p$ defined as follows: for any $\mathbb{P}\in {\cal P}^p( \Omega^{0,x})$ and $(\mathbb{P}^n)_{n\in \mathbb{N}}\subseteq {\cal P}^p( \Omega^{0,x})$, we have 
\begin{align}\label{eq:top.tau.wass}
\mathbb{P}^n \xrightarrow[]{\tau_p} \mathbb{P}\quad \mbox{as}\;\; n\rightarrow \infty\quad \Leftrightarrow \quad  \lim_{n\rightarrow \infty }\mathbb{E}^{\mathbb{P}^n}\left[\xi\right] =  \mathbb{E}^{\mathbb{P}}\left[\xi\right]\quad \mbox{for all} \;\; \xi \in C_p(\Omega^{0,x};\mathbb{R}).
\end{align}
Note that $\tau_p$ is the topology induced by the $p$-Wasserstein distance, see, e.g., \cite{villani2008optimal}.

Recalling the set $\mathbb{S}^d$ of all symmetric $d\times d$ matrices, denote by $\mathbb{S}^d_+\subset \mathbb{S}^d$ the subset of all positive semi-definite matrices. Let ${\cal P}_{\operatorname{sem}}$ be the set of all ${\mathbb{P}}\in {\cal P}(\Omega^{0,x})$ such that $X$ is a semimartingale on $(\Omega^{0,x}, {\cal F}^{0,x},\mathbb{F}^{X},{\mathbb{P}})$. Moreover, let ${\cal P}_{\operatorname{sem}}^{\operatorname{ac}}$ be the subset of all ${\mathbb{P}}\in {\cal P}_{\operatorname{sem}}$ such that $\mathbb{P}$-a.s.
$
B^{{\mathbb{P}}}\ll ds$ and 
$C^{{\mathbb{P}}}\ll ds,
$
where $B^{{\mathbb{P}}}$ and $C^{{\mathbb{P}}}$ denote the finite variation part and quadratic covariation of the local martingale part of $X$ under ${\mathbb{P}}$ having values in $\mathbb{R}^d$ and $\mathbb{S}^d_+$, respectively (i.e., the first and second characteristics of $X$) and are absolutely continuous with respect to $ds$ on~$[0,T]$.

Furthermore, we fix a mapping $\mathbb{S}^d_+\ni A\rightarrow A^{\frac{1}{2}}\in \mathbb{R}^{d\times d}$ so that it is Borel measurable and satisfies $A^{\frac{1}{2}}(A^{\frac{1}{2}})^\top=A$ for all $A\in \mathbb{S}_+^d$, (see, e.g., \cite[Remarks 1.1\;\&\;2.1]{possamai2018stochastic}). 

\subsection{Weak formulation and dynamic programming principle}  
For any $\varepsilon\geq 0$ and $(t,x)\in [0,T)\times \mathbb{R}^d$, define by
\begin{align}\label{eq:uncertain.weak}
{\cal P}^\varepsilon ({t,x}):=\left\{ {\mathbb{P}}\in {\cal P}_{\operatorname{sem}}^{\operatorname{ac}} \, \Big|\,
\begin{array}{l}
\mathbb{P}(X_{t\wedge \cdot}=x)=1;\;\;(b_s^{{\mathbb{P}}},(c_s^{\mathbb{P}})^{\frac{1}{2}})\in \mathcal{B}^\varepsilon\\
\mbox{for ${\mathbb{P}}\otimes ds$-almost every $(\omega,s)\in\Omega^{0,x}\times [t,T]$}
\end{array}
\right\}
\end{align}
where we recall that ${\cal B}^\varepsilon$ is given in \eqref{eq:uncertain.char}.

In particular, under any $\mathbb{P}\in{\cal P}^\varepsilon ({t,x})$, the semimartingale $X$ is constant (taking the value $x$) up to time $t$ and after that time its differential characterstics $b^{{\mathbb{P}}}:=\frac{dB^{{\mathbb{P}}}}{ds}$, $c^{{\mathbb{P}}}:=\frac{dC^{{\mathbb{P}}}}{ds}$ satisfy the value constraint as the set $\mathcal{B}^\varepsilon$.

Moreover, recall the function $f$ given in \eqref{eq:nonlinear.pde}. For any $\varepsilon\geq 0$, we define the value function $v^\varepsilon_{\operatorname{weak}}: [0,T]\times \mathbb{R}^d\ni (t,x)\rightarrow v^\varepsilon_{\operatorname{weak}}(t ,x )\in \mathbb{R}$ by setting {for} every $(t,x)\in[0,T)\times \mathbb{R}^d$,
\begin{align}\label{dfn:weak}
v^\varepsilon_{\operatorname{weak}}(t ,x):= \sup_{{\mathbb{P}}\in {\cal P}^{\varepsilon} ({t,x})} \mathbb{E}^{{\mathbb{P}}}\left[f\left(X_T\right)\right]
\end{align}
and $v^\varepsilon_{\operatorname{weak}}(T,\cdot):= f(\cdot)$ on $\mathbb{R}^d$. 

\begin{lem}\label{lem:priori_weak}
For every $p\geq 1$ and $\varepsilon\geq 0$, there is a constant ${C}_{p,\varepsilon}>0$ such that for every $(t,x)\in [0,T)\times \mathbb{R}^d$ and $s\in[t,T]$,
\[
\sup_{\mathbb{P}\in {\cal P}^\varepsilon(t,x)}\mathbb{E}^{\mathbb{P}}\left[\sup_{t\leq u \leq s}\left| X_u-x \right|^p \right] \leq {C}_{p,\varepsilon}\Big((s-t)^{p/2} + (s-t)^p \Big).
\]
\end{lem}
\begin{proof}
Fix $\varepsilon\geq 0$, $(t,x)\in [0,T)\times \mathbb{R}^d$, and $s\in[t,T]$, and let $\mathbb{P}\in {\cal P}^\varepsilon(t,x)$. Then under $\mathbb{P}$, the process $X$ has the canonical representation
$
X_s=x+\int_t^s b_r^{\mathbb{P}}dr+M^{\mathbb{P},t}_s,
$
where $({M}_s^{\mathbb{P},t})_{s\in [t,T]}$  denotes $(\mathbb{F}^X,\mathbb{P})$-local martingale part of $({X}_s)_{s\in [t,T]}$ satisfying $M_t^{\mathbb{P},t}=0$ with its differential characteristic $c^{\mathbb{P}}$ satisfying the constraint as $\mathcal{B}^\varepsilon$; see~\eqref{eq:uncertain.weak}. 

By Jensen's inequality and the definition of ${\cal P}^\varepsilon(t,x)$,
\[
\mathbb{E}^{\mathbb{P}}\left[\sup_{t\leq u \leq s}\Big|\int_t^u b_r^{\mathbb{P}} dr \Big|^p\right] \leq (s-t)^{p-1} \mathbb{E}^{\mathbb{P}}\left[\int_t^s |b_r^{\mathbb{P}}|^p dr \right]\leq 2^p(\varepsilon^p + |b^o|^p) (s-t)^p,
\]
where we use the elementary inequality $(a+b)^p\leq 2^p (a^p+b^p)$ for all $a,b\geq 0$.

Moreover, by the Burkholder-Davis-Gundy inequality and the elementary inequality $\|AB\|_{\operatorname{F}}\leq \|A\|_{\operatorname{F}}\|B\|_{\operatorname{F}}$ for all $A,B\in \mathbb{R}^d$,
\begin{align}
\begin{aligned}
	\mathbb{E}^{\mathbb{P}}\left[\sup_{t\leq u \leq s}\left| M_u^{\mathbb{P},t} \right|^p\right] &\leq c_{{\rm BDG},p} \mathbb{E}^{\mathbb{P}}\left[ \left(\int_t^s \| (c_r^{\mathbb{P}})^{\frac{1}{2}} \|_{\operatorname{F}}^2 ds\right)^{p/2} \right] \\
	&\leq  c_{{\rm BDG},p} \left( 2^2(\varepsilon^2 +\|\sigma^o\|_{\operatorname{F}}^2 ) \right)^{p/2} (s-t)^{{p}/{2}}.
\end{aligned}
\end{align}
Our claim follows by using again the inequality $(a+b)^p\leq 2^p (a^p+b^p)$ for all $a,b\geq 0$ and setting ${C}_{p,\varepsilon}:=2^p\{2^p(\varepsilon^p + |b^o|^p) + c_{{\rm BDG},p} ( 2^2(\varepsilon^2 +\|\sigma^o\|_{\operatorname{F}}^2 ) )^{p/2} \}$. 
\end{proof}

\begin{rem}\label{rem:inclusion0}
Lemma \ref{lem:priori_weak} implies that ${\cal P}^\varepsilon(t,x)$ is a subset of ${\cal P}^p(\Omega^{0,x})$ for every $\varepsilon\geq 0$ and $p\geq 1$; see \eqref{eq:p.msr.set} and \eqref{eq:uncertain.weak}.
\end{rem}

 Next, we present the dynamic programming principle for $v^\varepsilon_{\operatorname{weak}}$ in \eqref{dfn:weak} based on \cite{karoui2013capacities,nutz2013constructing}, together with its regularity property. 
\begin{lem}\label{lem:dynamic.value} Suppose that Assumption \ref{as:objective} is satisfied, let $\varepsilon\geq 0$, and let $v^\varepsilon_{\operatorname{weak}}$ be defined in~\eqref{dfn:weak}. Moreover, let $(t,x)\in [0,T)\times \mathbb{R}^d$. Then, the following hold: 
\begin{itemize}
\item [(i)] For any $\mathbb{F}^{X}$-stopping time $\tau$ taking values in $[t,T]$ 
\begin{align}\label{eq:dpp}
	v^\varepsilon_{\operatorname{weak}}(t,x)=  \sup_{{\mathbb{P}}\in {\cal P}^{\varepsilon}({t,x})} \mathbb{E}^{\mathbb{P}}\left[v^\varepsilon_{\operatorname{weak}}(\tau,X_{\tau})\right].
\end{align}
\item [(ii)] $v^\varepsilon_{\operatorname{weak}}$ is jointly continuous.
\end{itemize}
\end{lem}
\begin{proof}
We start by proving the statement (i). We claim that the set 
\begin{align}\label{eq:borel_map}
\left\{(\omega,t,{\mathbb{P}})\in \Omega^{0,x}\times [0,T]\times {\cal P}(\Omega^{0,x})\;|\;{\mathbb{P}}\in {\cal P}^\varepsilon({t,\omega_t})\right\}
\end{align}
is Borel. Indeed, since $\mathcal{B}^\varepsilon$ is Borel (see \eqref{eq:uncertain.char}) and the map $\mathbb{S}^d_+\ni A\rightarrow A^{\frac{1}{2}}\in\mathbb{R}^{d\times d}$ is Borel-measurable (see Section \ref{sec:prelimi_weak}),  the same arguments presented for the proof of \cite[Lemma 3.1]{fadina2019affine} using the existence of a Borel-measurable map from $\Omega^{0,x}\times [0,T]\times {\cal P}(\Omega^{0,x})$ to the differential characteristics of $X$ given in \cite[Theorem 2.6]{NeufeldNutz2014} ensure the claim to hold.

Furthermore, from \cite[Theorem 2.1]{neufeld2017nonlinear}, the following stability properties of ${\cal P}^\varepsilon({t,x})$ also hold: for any ${\mathbb{P}}\in{\cal P}^\varepsilon({t,x})$ and $\mathbb{F}^{X}$-stopping time $\tau$ having values in $[t,T]$, \vspace{-0.1em}
\begin{itemize}[leftmargin=1.9em]
\item [(a)] There is a set of conditional probability measures $({\mathbb{P}}_\omega)_{\omega \in \Omega^{0,x}}$ of ${\mathbb{P}}$ with respect to ${\cal F}_\tau^{X}$ such that ${\mathbb{P}}_\omega \in {\cal P}^\varepsilon({\tau(\omega),\omega_{\tau(\omega)}})$ for $\mathbb{P}$-almost all $\omega\in \Omega^{0,x}$;
\item [(b)] If there is a set of probability measures $({\mathbb{Q}}_\omega)_{\omega \in \Omega^{0,x}}$ such that ${\mathbb{Q}}_\omega \in {\cal P}^\varepsilon({\tau(\omega),\omega_{\tau(\omega)}})$ for ${\mathbb{P}}$-almost all $\omega\in\Omega^{0,x}$, and the map $\omega \rightarrow {\mathbb{Q}}_\omega$ is ${\cal F}_\tau^{X}$-measurable, then the probability~measure 
$
{\mathbb{P}}\otimes {\mathbb{Q}}(\cdot):=\int_{\Omega^{0,x}} {\mathbb{Q}}_\omega(\cdot){\mathbb{P}}(d\omega)
$
is an element of ${\cal P}^\varepsilon(t,x)$. 
\end{itemize}
Therefore, 
an application of \cite[Theorem 2.1]{karoui2013capacities} (see also \cite[Theorem~2.3]{nutz2013constructing}) 
ensures \eqref{eq:dpp} to hold.

\vspace{0.5em} Now let us prove (ii). Since $v_{\operatorname{weak}}^\varepsilon(T,\cdot)=f(\cdot)$ is continuous (by Assumption \ref{as:objective}), we can and do consider arbitrary $(t,x)\in[0,T)\times \mathbb{R}^d$. The continuity of $v^\varepsilon_{\operatorname{weak}}(t,\cdot)$ follows from the definition of $v^\varepsilon_{\operatorname{weak}}$ given in \eqref{dfn:weak}. 
Indeed, 
since for every $x,y\in \mathbb{R}^d$
\[
v^\varepsilon_{\operatorname{weak}}(t ,y)= \sup_{{\mathbb{P}}\in {\cal P}^{\varepsilon} ({t,y})} \mathbb{E}^{{\mathbb{P}}}\Big[f\left(X_T\right)\Big] = \sup_{{\mathbb{P}}\in {\cal P}^{\varepsilon} ({t,x})} \mathbb{E}^{{\mathbb{P}}}\Big[f\left(X_T+y-x\right)\Big],
\]
by Remark \ref{rem:f.nablaf.integrable} (with the constants $p\geq 1$ and  $c_1>0$) and the elementary property $(a+b)^p\leq 2^p(a^p+b^p)$ for all $a,b\geq0$, we have that 
\begin{align*}
\begin{aligned}
	&|v^\varepsilon_{\operatorname{weak}}(t ,y)-v^\varepsilon_{\operatorname{weak}}(t ,x)|\\
	&\leq \sup_{{\mathbb{P}}\in {\cal P}^{\varepsilon} ({t,x})} \mathbb{E}^{{\mathbb{P}}}\Big[\big|f\left(X_T+y-x\right) -f(X_T) \big|\Big]\\
	& \leq \sup_{{\mathbb{P}}\in {\cal P}^{\varepsilon} ({t,x})} \left\{ \mathbb{E}^{{\mathbb{P}}}\Big[\big| \nabla_x^\top f(X_T) \big| \Big] \cdot |y-x| + c_1  2^p \Big(1+|y-x|^p+ \mathbb{E}^{{\mathbb{P}}}\Big[|X_T|^p\Big]\Big)\cdot |y-x|^2\right\}.
\end{aligned}
\end{align*}
From Lemma \ref{lem:priori_weak} together with the polynomial growth property of $\nabla_x f$, we hence have that there is a constant $c_{3}>0$ (that depends on $p$, $\varepsilon$, $x$, but not on $t$) 
such that
\begin{align}\label{eq:conti_x_weak}
|v^\varepsilon_{\operatorname{weak}}(t ,y)-v^\varepsilon_{\operatorname{weak}}(t ,x)| \leq c_{3} \left(|y-x|+ |y-x|^{p+2}\right),
\end{align}
where we further emphasize that the above estimate holds for every $t\in[0,T)$ and $x,y\in\mathbb{R}^d$.

Now we claim that $v^\varepsilon_{\operatorname{weak}}(\cdot,x)$ is continuous. To that end, fix any $0\leq u \leq T-t$. By the dynamic programming principle of $v^\varepsilon_{\operatorname{weak}}$ (see Lemma \ref{lem:dynamic.value}\;(i)), the following~holds 
\[
v^\varepsilon_{\operatorname{weak}}(t,x)=  \sup_{{\mathbb{P}}\in {\cal P}^{\varepsilon}({t,x})} \mathbb{E}^{\mathbb{P}}\Big[v^\varepsilon_{\operatorname{weak}}(t+u,X_{t+u})\Big].
\]
Hence, we use again Lemma \ref{lem:priori_weak} together with the estimates given in \eqref{eq:conti_x_weak} to have\;that
\begin{align*}
\begin{aligned}
	|v^\varepsilon_{\operatorname{weak}}(t,x)-v^\varepsilon_{\operatorname{weak}}(t+u,x)| & \leq   \sup_{{\mathbb{P}}\in {\cal P}^{\varepsilon}({t,x})} \mathbb{E}^{\mathbb{P}}\left[\Big| v^\varepsilon_{\operatorname{weak}}(t+u,X_{t+u})- v^\varepsilon_{\operatorname{weak}}(t+u,x) \Big|\right] \\
	&\leq c_3 \sup_{{\mathbb{P}}\in {\cal P}^{\varepsilon}({t,x})} \left( \mathbb{E}^{\mathbb{P}}\Big[ |X_{t+u}-x|\Big]+ \mathbb{E}^{\mathbb{P}}\Big[ |X_{t+u}-x|^{p+2} \Big] \right)\\
	& \leq  c_3 \cdot \left({C}_{1,\varepsilon} \big(u^{1/2} + u \big)+ {C}_{p,\varepsilon} \big(u^{\frac{p+2}{2}} + u^{p+2} \big)\right),
\end{aligned}
\end{align*}
where $C_{1,\varepsilon},{C}_{p,\varepsilon}$ are the constant (with exponents $1,p$) appearing in Lemma \ref{lem:priori_weak} (and in particular do not depend on $x$). Combined with \eqref{eq:conti_x_weak}, this ensures that $v^\varepsilon_{\operatorname{weak}}$ is jointly continuous.
\end{proof}

\subsection{Proof of Proposition \ref{pro:main}}\label{sec:proof:pro:main}

Let us introduce the notion of viscosity / classical solution of \eqref{eq:nonlinear.pde} and \eqref{eq:linear.pde.1stdev}. To that end, we introduce the following function spaces: for any $t\in[0,T)$
\begin{itemize}[leftmargin=2em]
\item [$\cdot$] $C^{1,2}([t,T)\times \mathbb{R}^d;\mathbb{R})$ is the set of all real-valued functions on $[t,T)\times\mathbb{R}^d$ which are continuously differentiable on $[t,T)$ and twice continuously differentiable on $\mathbb{R}^d$;
\item [$\cdot$] $C_b^{2,3}([t,T)\times \mathbb{R}^d;\mathbb{R})$ is the set of all real-valued functions on $[t,T)\times\mathbb{R}^d$ which have bounded continuous derivatives up to the second and third order on $[t,T)$ and $\mathbb{R}^d$, respectively.
\end{itemize}

\begin{dfn}[Viscosity solution\;(see~\cite{crandall1992user,fleming2006controlled})]\label{dfn:viscosity_pde}
Fix any $\varepsilon\geq 0$. We call an upper semicontinuous function $v^\varepsilon:[0,T]\times \mathbb{R}^d\rightarrow \mathbb{R}$ a viscosity subsolution of \eqref{eq:nonlinear.pde} if $v^\varepsilon(T,\cdot)\leq f(\cdot)$ on $\mathbb{R}^d$ and 
\[
-\partial_t \varphi(t,x)- \sup_{(b,\sigma)\in \mathcal{B}^\varepsilon}\left\{\frac{1}{2} \operatorname{tr} \big(\sigma \sigma^\top D^2_{xx}\varphi(t,x) \big)+\langle b, \nabla_x \varphi (t,x) \rangle  \right\} \leq 0
\]
whenever\footnote{It is well known that restricting test functions to either $C^{2,3}_b([0,T)\times \mathbb{R}^d; \mathbb{R})$ or $C^{1,2}([0,T)\times \mathbb{R}^d; \mathbb{R})$ does not affect the definition of a viscosity solution. In this work, we adopt $C^{2,3}_b([0,T)\times \mathbb{R}^d; \mathbb{R})$ to align with results in the literature stated in this form.} $\varphi \in C_b^{2,3}([0,T)\times\mathbb{R}^d;\mathbb{R})$ is such that $\varphi\geq v^\varepsilon$ on $[0,T)\times \mathbb{R}^d$ and $\varphi(t,x)=v^\varepsilon(t,x)$. In a similar manner, the notion of a viscosity supersolution can be defined by reversing the inequalities and replacing upper semicontinuity with lower semicontinuity. Finally, we call a continuous function $v^\varepsilon$ from $[0,T]\times \mathbb{R}^d$ to $\mathbb{R}$ a viscosity solution if it is both sub- and supersolution of \eqref{eq:nonlinear.pde}.	
\end{dfn}

\begin{dfn}[Classical solution]\label{dfn:strong_pde}
Fix $t\in[0,T)$ and $i=1,\dots,d$. We call a continuous function $w^i:[t,T]\times \mathbb{R}^d\rightarrow\mathbb{R}$ a classical solution of \eqref{eq:linear.pde.1stdev} if it is in $C^{1,2}([t,T)\times \mathbb{R}^d;\mathbb{R})$ and satisfies \eqref{eq:linear.pde.1stdev}. 
\end{dfn}

\begin{lem}\label{lem:weak_viscosity}
Suppose that Assumptions \ref{as:objective} and \ref{as:comparison} are satisfied and let $\varepsilon\geq 0$. Then $v^\varepsilon_{\operatorname{weak}}:[0,T]\times \mathbb{R}^d\rightarrow \mathbb{R}$ defined in \eqref{dfn:weak} is a unique viscosity solution of \eqref{eq:nonlinear.pde} satisfying \eqref{eq:growth_as}.
\end{lem}
\begin{proof}
By Lemma \ref{lem:dynamic.value}, $v^\varepsilon_{\operatorname{weak}}$ satisfies the dynamic programming principle and is jointly continuous. Hence, the same arguments as presented for the proof of \cite[Proposition 5.4]{neufeld2017nonlinear} ensure that $v^\varepsilon_{\operatorname{weak}}$ is a unique viscosity solution of \eqref{eq:nonlinear.pde}. Furthermore, as the function $f$ has at most polynomial growth (see Remark \ref{rem:f.nablaf.integrable}), Lemma~\ref{lem:priori_weak} ensures that $v^\varepsilon_{\operatorname{weak}}$ has polynomial growth with respect to $x\in \mathbb{R}^d$ for all $t\in[0,T]$. This implies that $v^\varepsilon_{\operatorname{weak}}$ satisfies \eqref{eq:growth_as} with some $C>0$. Hence by Assumption \ref{as:comparison}, $v^\varepsilon_{\operatorname{weak}}$ is the unique viscosity solution satisfying \eqref{eq:growth_as}. 
\end{proof}

\begin{proof}[Proof of Proposition \ref{pro:main}] 
The statement (i) follows directly from Lemma \ref{lem:weak_viscosity}. Now let us prove\;(ii). Note that $\nabla_x f$ has at most polynomial growth (see Remark\;\ref{rem:f.nablaf.integrable})  and $(b^o,\sigma^o)$ are constant. Furthermore, $\sigma^o$ is non-degenerate (see Assumption\;\ref{as:sigma.inverse}). Hence, an application of \cite[Theorem~5.7.6 \& Remark~5.7.8]{KS1991} (see also \cite[Theorem 4.32]{Krylov1999}) 
ensures the existence of a classical solution of \eqref{eq:linear.pde.1stdev}. The uniqueness of the solution with polynomial growth is guaranteed by \cite[Corollary~6.4.4]{friedman1975stochastic}.
\end{proof}

\subsection{Strong formulation and its equivalence}\label{sec:StrongWeak}
In this section, we construct a set of probability measures corresponding to a strong formulation of the nonlinear Kolmogorov PDE given in \eqref{eq:nonlinear.pde}. 

Recall the process $X^{t,x;b,\sigma}$ defined on $[t,T]$ (given in \eqref{eq:ito.semi.time}) and denote by $(x\oplus_t X^{t,x;b,\sigma})$ the constant concatenation of $X^{t,x;b,\sigma}$ defined on $[0,T]$, i.e. 
\begin{align}\label{eq:ito.semi.time.concat}
(x\oplus_t X^{t,x;b,\sigma})_s:=x {\bf 1}_{\{s \in [0,t)\}}+ X^{t,x;b,\sigma}_s {\bf 1}_{\{s\in [t,T]\}}.
\end{align}

Then using the set ${\cal C}^\varepsilon(t)$ given in \eqref{eq:char.strong}, we define a set of (push-forward) probability measures as follow: for any $\varepsilon\geq 0$ and $(t,x)\in[0,T)\times \mathbb{R}^d$
\begin{align}\label{eq:uncertain.strong1}
{\cal Q}^\varepsilon(t,x):={\cal Q} \left({t,x};{\cal C}^{\varepsilon} \right)=\left\{\mathbb{P}_0^t\circ \left(x\oplus_t {X}^{t,x;b,\sigma}  \right)^{-1}\;\Big|
\;\mbox{$(b,\sigma)\in {\cal C}^{\varepsilon}(t)$}\right\}\subseteq {\cal P}(\Omega^{0,x}).
\end{align}

\begin{rem}\label{rem:inclusion}
By the definition of $(x\oplus_t X^{t,x;b,\sigma})$ given in \eqref{eq:ito.semi.time.concat}, ${\cal Q}^\varepsilon(t,x)$ is a subset of ${\cal P}^{\varepsilon}(t,x)$ for every $\varepsilon\geq 0$; see \eqref{eq:uncertain.weak} for the definition.
\end{rem}

\begin{pro}\label{pro:weak.dense} 
Suppose that Assumption \ref{as:sigma.inverse} is satisfied. Let $\varepsilon < \lambda_{\min}(\sigma^o)$ (see Remark~\ref{rem:sigma.inverse}) and $(t,x)\in[0,T)\times \mathbb{R}^d$. Moreover, let ${\cal P}^\varepsilon ({t,x})$ and ${\cal Q}^\varepsilon \left({t,x}\right)$ be defined in \eqref{eq:uncertain.weak} and \eqref{eq:uncertain.strong1}, respectively. Then, there exists ${\cal Q}^{\varepsilon}_{\operatorname{sub}}(t,x)\subseteq {\cal Q}^\varepsilon \left({t,x}\right)$ such that its convex hull is a dense subset of ${\cal P}^\varepsilon ({t,x})$ with respect to the $\tau_p$-topology for all $p\geq1$. 
\end{pro}

A variant of Proposition \ref{pro:weak.dense} already appears in \cite{DNS2012}, see Proposition 3.5 therein and its proof.
The difference to our setting is that in  \cite{DNS2012} the authors consider volatility only (i.e.\ corresponding to $b=0$) and  weak convergence instead of the $\tau_p$-topology (i.e.\ consider only bounded test functions instead of polynomially growing functions, see \eqref{eq:top.tau.wass}).
However, the proof presented in \cite{DNS2012} can be copied line by line: adding drift does not require any different arguments at all, while the extension from weak convergence to $\tau_p$-convergence only requires to change the application of the Hahn-Banach theorem in step 3 of the proof of  \cite[Proposition 3.5]{DNS2012} from the space of bounded continuous functions to those with polynomial growth.
For convenience of the reader, we repeat the proof presented in \cite{DNS2012} adjusted to the present setting in Appendix \ref{sec:supple}. 

Recall the function $f:\mathbb{R}^d\rightarrow \mathbb{R}$ given in \eqref{eq:nonlinear.pde} and the canonical process $X=(X_s)_{s\in[0,T]}$ defined on $(\Omega^{0,x},{\cal F}^{0,x},\mathbb{F}^{X})$. For any $\varepsilon\geq 0$, we define the value function $v_{\operatorname{strong}}^\varepsilon:[0,T]\times \mathbb{R}^d\rightarrow \mathbb{R}$ by setting for every $(t,x)\in [0,T)\times \mathbb{R}^d$,
\begin{align}\label{dfn:strong}
v^\varepsilon_{\operatorname{strong}}(t,x):= \sup_{\mathbb{P}\in \mathcal{Q}^{\varepsilon}(t,x)} \mathbb{E}^{{\mathbb{P}}}\left[f\left(X_T\right)\right]=  \sup_{(b,\sigma)\in \mathcal{C}^{\varepsilon}(t)}\mathbb{E}^{\mathbb{P}_0^t}\left[f\left(X^{t,x;b,\sigma}_T\right)\right]
\end{align}
and $v^\varepsilon_{\operatorname{strong}}(T,\cdot):=f(\cdot)$ on $\mathbb{R}^d$. We call this the `strong formulation' of modeling uncertainty of $X$, which will turn out to be equivalent to the weak formulation $v^\varepsilon_{\operatorname{weak}}$ in the next proposition. 

\begin{pro}\label{pro:equiv.weak.strong}
Suppose that Assumptions \ref{as:objective}, \ref{as:sigma.inverse}, and \ref{as:comparison} are satisfied and let $v^\varepsilon_{\operatorname{weak}}$ and $v^\varepsilon_{\operatorname{strong}}$ be defined in \eqref{dfn:weak} and \eqref{dfn:strong}, respectively. Then the following  equality holds: for any $\varepsilon < \lambda_{\min}(\sigma^o)$ and $(t,x)\in [0,T]\times \mathbb{R}^d$,
\[
v^\varepsilon_{\operatorname{strong}}(t,x)=v^\varepsilon_{\operatorname{weak}}(t,x).
\]
\end{pro}
\begin{proof}
Recall the sets ${\cal P}^{\varepsilon}({t,x})$ and ${\cal Q}^{\varepsilon}({t,x})$ defined in \eqref{eq:uncertain.weak} and \eqref{eq:uncertain.strong1}. Denote by ${\cal Q}^\varepsilon_{\operatorname{sub}}(t,x)$ the subset of ${\cal Q}^\varepsilon(t,x)$ such that its' convex hull is a dense subset of ${\cal P}^\varepsilon(t,x)$ with respect to the $\tau_p$-topology for all $p\geq 1$ (see Proposition \ref{pro:weak.dense}). 
Furthermore, ${\cal P}^\varepsilon(t,x)$ is a subset of  ${\cal P}^p(\Omega^{0,x})$ for every $p\geq 1$ (see Remark \ref{rem:inclusion0}).
Since for every $\xi \in C_p(\Omega^{0,x};\mathbb{R})$ the map ${\cal P}^p(\Omega^{0,x})\ni \mathbb{P}\rightarrow \mathbb{E}^{{\mathbb{P}}}\left[\xi\right]$ is continuous and linear, it follows that
$
\sup_{{\mathbb{P}}\in {\cal P}^{\varepsilon}({t,x})} \mathbb{E}^{{\mathbb{P}}}\left[\xi\right]  = \sup_{{\mathbb{P}}\in {\cal Q}^{\varepsilon}({t,x})} \mathbb{E}^{{\mathbb{P}}}\left[\xi\right].
$
Therefore, as the function $f$ has at most polynomial growth (see Remark \ref{rem:f.nablaf.integrable}), 
\begin{align*}
v^\varepsilon_{\operatorname{weak}}(t,x)=\sup_{{\mathbb{P}}\in {\cal P}^{\varepsilon} ({t,x})} \mathbb{E}^{{\mathbb{P}}}\left[f(X_T)\right]
= \sup_{{\mathbb{P}}\in {\cal Q}^{\varepsilon}({t,x})} \mathbb{E}^{{\mathbb{P}}}\left[f(X_T)\right]  =v^\varepsilon_{\operatorname{strong}}(t,x).
\end{align*}
This completes the proof.
\end{proof}

\section{Proof of Theorem \ref{thm:MC}} For notational simplicity, throughout this section, we set $\mathbb{E}_{X_t^o=0}[\cdot]:=\mathbb{E}[\cdot|X_t^o=0]$ for every $t\in[0,T]$.
We adopt the same notation for the conditional variance and the conditional expectation/variance given $\widetilde{X}^o$.
\begin{rem}
\label{rem:f.hessian.integrable}	
If  $f$ satisfies the assumptions imposed in Theorem \ref{thm:MC} (in particular, its third derivatives having at most polynomial growth of order $\widetilde{\alpha}$ for some $\widetilde{\alpha}\geq 1$), $D_x^2f$, $\nabla_x f$, and $f$ are continuous and have at most polynomial growth of order $\widetilde \alpha+1$, $\widetilde \alpha+2$, and $\widetilde \alpha+3$, respectively (see Remark\;\ref{rem:f.nablaf.integrable}). Furthermore, there is $\widetilde c_f$ such that $|\nabla_x f(y)-\nabla_xf(\tilde y)|\leq \widetilde c_f(1+|y|^{\widetilde \alpha+1}+|\tilde y|^{\widetilde \alpha+1})|y-\tilde y|$ 
for every $y,\tilde y \in \mathbb{R}^d$. 
\end{rem}


\begin{lem}\label{lem:sup_est} 
Let $g:[0,T]\times\mathbb{R}^d\mapsto \mathbb{R}$  be Borel s.t.\ $|g(t,x)|\leq c_g(1+|x|^{\alpha})$  for every $(t,x)\in[0,T]\times \mathbb{R}^d$ for some $c_g, \alpha \geq 1$. 
Let $X_t$ and $\widetilde X_t^o$ be the processes appearing in \eqref{eq:Feyn_1}. 
Then there is a constant $C$ depending only on $\alpha,T,d,c_g,b^0,\sigma^0$ such that 
\begin{itemize}
	\item [(i)] $
	\mathbb{E}_{X_t^o=0}[\sup_{u\in [s_1,s_2]}~|X_u^o-X_{s_1}^o|^2]^{\frac{1}{2}}\leq  C (s_2-s_1)^{{1}/{2}}$,
	\item[(ii)] $
	\mathbb{E}_{X_t^o=0}[\sup_{s\in [t,T]}~|g(s,x+X_s^o)|^2]^\frac{1}{2}\leq C(1+|x|^{ \alpha})$,
	\item[(iii)] $
	\mathbb{E}_{X_t^o=0}[\sup_{s\in [t,T]}~|\mathbb{E}_{\widetilde{X}_s^o=0}[g(s,x+X_s^o+\widetilde{X}_T^o)]|^2]^\frac{1}{2}\leq C\big(1+|x|^{ \alpha}\big)$.
\end{itemize}
\end{lem}
%
\begin{proof} 
Since $X_s^o=b^o(s-t)+\sigma^o (W_s-W_t)$ for $s\in[t,T]$ under $X_t^o=0$, the BDG inequality shows that for every $p\geq 1$, 
\begin{align*}
		\mathbb{E}_{X_t^o=0}\bigg[\sup_{u\in [s_1,s_2]}~|X_u^o-X_{s_1}^o|^p\bigg]^{\frac{1}{p}}\hspace{-0.3em}
		&\leq |b^o|(s_{2}-s_1)+\mathbb{E}_{X_t^o=0}\bigg[\sup_{u\in[s_1,s_2]}\big|\sigma^o(W_u-W_{s_1})\big|^{p}\bigg]^\frac{1}{p} \\
		&\leq C_p (s_2-s_1)^{1/2}, 
\end{align*}
where $C_p$ depends only on $T, b^0,\sigma^0$ and the BDG-$p$-constant.
In particular, (ii) follows by choosing $p=2$.
Moreover, choosing $s_1=t$, we have that $\mathbb{E}_{X_t^o=0}[|X_s|^p]^{\frac{1}{p}} \leq C_p$.
Hence the statements (ii) and (iii) follow from the growth assumptions made on $g$.
%
\end{proof}

\begin{proof}[Proof of Theorem \ref{thm:MC}\;(i)]
Fix $(t,x)\in[0,T)\times \mathbb{R}^d$.
The constants $C_0,C_1,\dots$ in this proof depend on $t,T,x,d,\tilde{\alpha},b^0,\sigma^0$, but not on $M_0,N,M_1,M_2$.

\vspace{0.5em}
\emph{Step 1.} We first claim that $\|v^0(t,x)  -\widehat{v}^{0,(t,x)}_{M_0} \|_{L^2}\leq C_{0}/\sqrt{M_0}$.

Indeed, observe that  $v^0(t,x) = \mathbb{E}_{X_t^o=0}[f(x+X_T^o)]$ (see \eqref{eq:Feyn_1}).
Thus, since  $f(x+\xi_{T,t}^l),\dots ,f(x+\xi_{T,t}^{M_0})$ are i.i.d.\;random samples of $f(x+X_T^o)$ under $X_t^o=0$,
\begin{align*}
	\Big\|v^0(t,x)  -\widehat{v}^{0,(t,x)}_{M_0} \Big\|^2_{L^2}
	&= \mathbb{E}\bigg[\bigg| \mathbb{E}_{X_t^o=0}\Big[f(x+X_T^o)\Big]- \frac{1}{M_0} \sum_{l=1}^{M_0} f(x+\xi_{T,t}^{l}) \bigg|^2\bigg] \\
	&= \frac{ {\mathbb{V}\rm ar}_{X_t^o=0}[f(x+X_T^o)] }{M_0}
	\leq \frac{C_0}{M_0},
\end{align*}
where the  estimate on the variance follows  from Lemma \ref{lem:sup_est} because $f$ has polynomial growth of order at most $\tilde{\alpha}+3$ (see Remark \ref{rem:f.hessian.integrable})).
%
%

\vspace{0.5em}
\emph{Step 2.}
Set $v^{1,(t,x)}:=\mathbb{E}_{X_t^o=0}[\int_t^T |w(s, x+ X_s^{o})|ds ]$.
We claim that 
\begin{align}\label{eq:MC_est1}
	\left\| v^{1,(t,x)}- \widehat{v}^{1,(t,x)}_{N,M_1,M_2} \right\|_{L^2}
	\leq C_1\Big(\frac{1}{\sqrt N}+ \frac{1}{\sqrt M_1}+\frac{1}{\sqrt{M_2}}\Big).
\end{align}

The proof of \eqref{eq:MC_est1} consists of several steps.

\vspace{0.5em}
{\it Step\;2a:} 
For every $N\in\mathbb{N}$, set 
$A^{t,x}_N:=\sum_{i=0}^{N-1}\mathbb{E}_{X_t^o=0}[|w(t_i,x+X_{t_i}^o)|]\Delta t$.
Then, by the triangle inequality,
\begin{align*}
	\begin{aligned}
		\left|v^{1,(t,x)}- A^{t,x}_N \right|
		&\leq \sum_{i=0}^{N-1} \int_{t_i}^{t_{i+1}}\mathbb{E}_{X_t^o=0} \Big[ \big|w\left(s, x+ X_s^{o}\right)-w\left(t_i, x+ X_{t_i}^{o}\right) \big| \Big]ds
	\end{aligned}
\end{align*}
and clearly $\Big|w\left(s, x+ X_s^{o}\right)-w\left(t_i, x+ X_{t_i}^{o}\right) \Big|\leq {\rm I}_s^i + {\rm II}_s^i$ where
\begin{align*}
	{\rm I}_s^i := \big|w\left(s, x+ X_s^{o}\right)-w\left(s, x+ X_{t_i}^{o}\right) \big| ,
	\quad {\rm II}_s^i:= \big|w\left(s, x+ X_{t_i}^{o}\right)-w\left(t_i, x+ X_{t_i}^{o}\right) \big|.
\end{align*}
We estimate both terms separately.
%
Since $w(s,y)=\mathbb{E}_{\widetilde{X}_s^o=0}[\nabla_xf(y+\widetilde{X}_T^o)]$ for every $y\in \mathbb{R}^d$ (see \eqref{eq:Feyn_1}), we have that for every $i=1,\dots,N$,
\[\operatorname{I}_s^{i}=\Big| \mathbb{E}_{\widetilde X_s^o=0}\Big[\nabla_xf(x+X_s^o+\widetilde{X}_T^o)-\nabla_xf(x+X_{t_i}^o+\widetilde{X}_T^o)\Big] \Big|.\]
And since  $D^2_xf$ has polynomial growh of order $\tilde{\alpha}+1$ (see Remark \ref{rem:f.hessian.integrable}),  an application of the Cauchy-Schwartz inequality implies that
\begin{align*}
	\mathbb{E}_{X_t^o=0}[ \operatorname{I}_s^{i}]
	\leq  C_2\;\mathbb{E}_{  X_t^o,\widetilde{X}_t^o=0 } \bigg[\Big(1+\hspace{-1.2em}\sup_{s\in[t_i,t_{i+1}]}|x+X_s^o+\widetilde{X}_T^o|^{\widetilde \alpha+1}\Big)^2\bigg]^{\frac{1}{2}}\hspace{-0.2em}\mathbb{E}_{X_t^o=0}\Big[|X_s^o-X_{t_i}^o|^2\Big]^{\frac{1}{2}}.
\end{align*}
Thus, by Lemma  \ref{lem:sup_est}, $\mathbb{E}_{X^o_t=0}[ \operatorname{I}_s^{i}]\leq C_3/\sqrt N$.


%
%
%

Next, by Proposition \ref{pro:main}\;(ii) (see also Lemma\;\ref{lem:BSDE}), applying (in that order) It\^o's formula, Jensen's inequality, and the It\^o-isometry ensures that 
\begin{align*}
	\mathbb{E}_{X^o_t=0}[{\rm II}^i_s]
	&=\mathbb{E}_{X_t^o,\widetilde{X}_t^o=0 }\bigg[ \bigg|\int_{t_i}^s \operatorname{J}_xw\big(u,x+X_{t_i}^o+\widetilde{X}_u^o\big)\sigma^o d\widetilde W_u\bigg|\bigg]  \\
	&\leq  \mathbb{E}_{X_t^o,\widetilde{X}_t^o=0} \bigg[ \int_{t_i}^s\|\operatorname{J}_xw\big(u,x+X_{t_i}^o+\widetilde{X}_u^o\big)\sigma^o\|_{\operatorname{F}}^2 du  \bigg]^{\frac{1}{2}}.
\end{align*}
Finally, since ${\rm J}_x w(s,y) =  \mathbb{E}_{\widetilde{X}_t^o=0}[D^2_x f(y+ \widetilde{X}_T^o)]$ (see \eqref{eq:Feyn_2}) and $D_x^2 f$ has polynomial growth of order $\tilde{\alpha}+1$ (see Remark \ref{rem:f.hessian.integrable}), it follows from   Lemma \ref{lem:sup_est} that $\mathbb{E}_{X_t^o=0}[{\rm II}^i_s] \leq C_4 \sqrt{s-t_i} \leq C_5/\sqrt N$.

\vspace{0.5em}
\noindent	
{\it Step 2b:} For every $N,{M}_1\in \mathbb{N}$, set 
	$B_{N,M_1}^{t,x}:= \sum_{i=0}^{N-1}\frac{1}{M_1}\sum_{m=1}^{M_1} \big|w(t_i,x+\xi_{t_i,t}^m)\big| \Delta.$
Recall that $A^{t,x}$ was defined in Step 2a.
Then, by the triangle inequality,
%
\begin{align*}
	\big\|A^{t,x}_N-B_{N,M_1}^{t,x} \big\|_{L^2}&\leq \sum_{i=0}^{N-1} \bigg\| \mathbb{E}_{X_s^o=0}[|w(t_i,x+X_{t_i}^o)|]-\frac{1}{M_1}\sum_{m=1}^{M_1}\big|w(t_i,x+\xi_{t_i,t}^m) \big|\bigg\|_{L^2}\Delta t  \\
	&=\sum_{i=0}^{N-1} \sqrt{ \frac{ \mathbb{V}{\rm ar}_{X_s^o=0}[|w(t_i,x+X^o_{t_i})|]  }{M_1}  }\Delta t 
	\leq \frac{C_6}{\sqrt{M}_1},
\end{align*}
where we again used that $w(s,y)=\mathbb{E}_{X_s^o=0}[\nabla_xf(y+X_T^o)]$, the polynomial growth of $\nabla_x f$, and Lemma \ref{lem:sup_est} to bound the variance.

\vspace{0.5em}
\noindent	
{\it Step 2c:} For every $N,{M}_1,M_2\in \mathbb{N}$, let $\widehat{v}^{1,(t,x)}_{N,M_1,M_2}$ be given in \eqref{eq:MC_estimates} and recall that $B_{N,M_1}^{t,x}$ was defined in Step 2b. 
Then, by the triangle inequality 
\begin{align*}
	\Big\|B_{N,M_1}^{t,x}-\widehat{v}^{1,(t,x)}_{N,M_1,M_2}\Big\|_{L^2}
	&\leq \sum_{i=0}^{N-1} \frac{\Delta t}{M_1}\sum_{m=1}^{M_1} \sum_{k=1}^d D_{M_2}^{k,m,i}, \quad\text{where}\\
	D_{M_2}^{k,m,i}&:=\bigg\| |w_k(t_i,x+\xi_{t_i,t}^m)|-\Big|\frac{1}{M_2}\sum_{n=1}^{M_2}\partial_{x_k}f(x+\xi_{t_i,t}^m+\widetilde {\xi}^n_{T,t_i}) \Big|  \bigg\|_{L^2}.
\end{align*}
Note that $w_k(t_i,x+\xi_{t_i,t}^m)=\mathbb{E}_{\widetilde{X}_{t_i}^o=0}[\partial_{x_k}f(x+\xi_{t_i,t}^m+\widetilde{X}_T^o)]$ and $\widetilde \xi_{T,t_i}^1,\dots,\widetilde \xi_{T,t_i}^{M_2}$ are the i.i.d~random samples of $\widetilde X_{T}^o$ given $\widetilde X_{t_i}^o=0$, independent of $ \xi_{T,t_i}^m$.
Thus, by the tower property, 
\[ (D_{M_2}^{k,m,i})^2
=\mathbb{E}\bigg[ \mathbb{E} \bigg[ \frac{ {\rm \mathbb{V}ar}_{\widetilde{X}_{t_i}^o=0}[ |\partial_{x_k} f(x+ \xi^m_{t_i,t}  + \widetilde{X}_T^o) |] }{M_2} \, \bigg| \xi^m_{t_i,t} \bigg] \bigg]
\leq \frac{C_7}{M_2},\]
where the inequality follows from  Lemma \ref{lem:sup_est} and using the growth condition on $f$ (see Remark \ref{rem:f.hessian.integrable}).
Thus $\|B_{N,M_1}^{t,x}-\widehat{v}^{1,(t,x)}_{N,M_1,M_2}\|_{L^2}\leq C_8/\sqrt{M_2}$.

The combination of Step 2a-2c shows the claim in Step 2.


	\vspace{0.5em}
	\emph{Step 3.}
	Set $v^{2,(t,x)}:= \mathbb{E}_{X^o_t=0}[\int_t^T  \| \mathrm{J}_xw(s, x+X_s^{o} ) \sigma^o \|_{\operatorname{F}} ) ds]$.
	We claim that 
	\[
	\left\| v^{2,(t,x)}- \widehat{v}^{2,(t,x)}_{N,M_1,M_2} \right\|_{L^2}
	\leq C_9\Big(\frac{1}{\sqrt N}+ \frac{1}{\sqrt M_1}+\frac{1}{\sqrt{M_2}}\Big).
	\]
	Indeed, the proof follows from the same arguments as presented for Step 2.

	\vspace{0.5em}
	Combining the estimates established in Step 1-3 concludes the proof.
\end{proof}

\begin{proof}[Proof Theorem \ref{thm:MC}\;(ii)] We note that for every $N,M_0,M_1,M_2\in\mathbb{N}$, the complexity $\mathfrak{C}(N,M_0,M_1,$ $M_2)$ is given 
	\[
	\mathfrak{C}(N,M_0,M_1,M_2)=M_0d+N M_1(M_2+1) d+N M_1(M_2+1+d)d^2
	\]
	where 
	\begin{itemize}[leftmargin=2.5em]
		\item [$1.$] $M_0d$ is the number of ($d$-dimensional) random random samples for the Monte Carlo approximation used in $\widehat{v}^{0,(t,x)}_{M_0}= \frac{1}{M_0} \sum_{l=1}^{M_0} f(x+\xi_{T,t}^l)$;
		\item [$2.$] $N M_1(M_2+1) d$ is the sum of the number $NM_1M_2d$ of samples for nested Monte Carlo approximation and the number $NM_1d$  of Euclidean norm evaluations used in $\widehat{v}^{1,(t,x)}_{N,M_1,M_2}= \sum_{i=0}^{N-1}\frac{1}{M_1}\sum_{m=1}^{M_1}|\frac{1}{M_2}\sum_{n=1}^{M_2}\nabla_xf(x+\xi_{t_i,t}^m+\widetilde {\xi}^n_{T,t_i})|\Delta t$;
		\item [$3.$] $N M_1(M_2+d+1)d^2$ is the sum of: the number $NM_1M_2d^2$ of samples for the nested Monte-Carlo approximation, the number $NM_1d^3$ of the matrix multiplications, and the number $NM_1d^2$ of Frobenius norm evaluations used in 
		$\widehat{v}^{2,(t,x)}_{N,M_1,M_2}=\sum_{i=0}^{N-1}\frac{1}{M_1}\sum_{m=1}^{M_1}\|\frac{1}{M_2}\sum_{n=1}^{M_2}$ $D^2_xf(x+\xi_{t_i,t}^m+\widetilde {\xi}^n_{T,t_i})\sigma^o\|_{\operatorname{F}}\Delta t$.
	\end{itemize} 
	Hence, this ensures Theorem \ref{thm:MC}\;(ii) to hold. 
\end{proof}

\vspace{0.6cm}
\section*{Declarations}
\noindent 
\textbf{Conflict of interest:} The authors have no conflicts of interest to declare that are relevant to the content of
this article.

\newpage

\appendix
\section*{Appendix}
\section{Proof of Proposition  \ref{pro:weak.dense}}\label{sec:supple}
As explained in Section~\ref{sec:StrongWeak}, 
Proposition~\ref{pro:weak.dense} constitutes a minor generalization of the analog result presented in  \cite{DNS2012}.
In particular, the proof presented in \cite{DNS2012} can be copied line by line and requires only minor changes for our generalization, which are presented in detail here.

First, we introduce some notions, often employed in this section. Recalling the set $\mathcal{B}^\varepsilon$ given in \eqref{eq:uncertain.char}, we set for any $\varepsilon\geq 0$,
\begin{align}\label{eq:uncertain.char.2}
	\mathcal{B}^{\varepsilon,1}:=\{b:(b,\sigma)\in \mathcal{B}^\varepsilon\},\qquad \mathcal{B}^{\varepsilon,2}:=\{\sigma:(b,\sigma)\in \mathcal{B}^\varepsilon\}
\end{align}
so that $\mathcal{B}^\varepsilon= \mathcal{B}^{\varepsilon,1}\times \mathcal{B}^{\varepsilon,2}$ and denote by 
\begin{align}\label{eq:project}
	\Pi_{\mathcal{B}^{\varepsilon,1}}:\mathbb{R}^d\ni x \rightarrow \Pi_{\mathcal{B}^{\varepsilon,1}}(x)\in\mathcal{B}^{\varepsilon,1},\qquad
	\Pi_{\mathcal{B}^{\varepsilon,2}}:\mathbb{R}^{d\times d}\ni x \rightarrow  \Pi_{\mathcal{B}^{\varepsilon,2}}(x)\in\mathcal{B}^{\varepsilon,2}
\end{align}
the Euclidean projections into the convex, closed sets $\mathcal{B}^{\varepsilon,1}$ and $ \mathcal{B}^{\varepsilon,2}$ respectively. 

For every $n\in \mathbb{N}$ and $t\in[0,T)$, denote by $
t_k^{n}:= t+\frac{T-t}{n}k$ for $k=0,1,\dots,n$. Furthermore, for any $\mathbb{P}\in {\cal P}^{\operatorname{as}}_{\operatorname{sem}}$ denote by $b^{\mathbb{P}}=\frac{dB^{\mathbb{P}}}{ds}$ the first differential characteristics of $X$ under $\mathbb{P}$, and~by 
\begin{align}\label{eq:sqrt_2nd_char}
	\sigma^{\mathbb{P}}:= (c^{\mathbb{P}})^{\frac{1}{2}}
\end{align}
where $c^{\mathbb{P}}=\frac{dC^{\mathbb{P}}}{ds}$ is the second differential characteristics of $X$ under $\mathbb{P}$. Then we define by $b^{\mathbb{P},(n)}$ and $\sigma^{\mathbb{P},(n)}$ piecewise constant processes defined on $[t,T]$ such that 
\begin{align}\label{eq:discret3}
	\begin{aligned}
		b_s^{\mathbb{P},(n)}&:=  {\bf 1}_{\{s\in [t,t_1^{n}]\}}~{b}^{o}+\sum_{k=1}^{n-1} {\bf 1}_{\{s\in(t_{k}^{n},t_{k+1}^{n}]\}}~\Pi_{\mathcal{B}^{\varepsilon,1}}\Bigg[\frac{n}{T-t}\int_{t_{k-1}^{n}}^{t_k^{n}}{b}_s^{\mathbb{P}}ds\Bigg],\\
		\sigma_s^{\mathbb{P},(n)}&:=  {\bf 1}_{\{s\in [t,t_1^{n}]\}}~\sigma^{o}+\sum_{k=1}^{n-1} {\bf 1}_{\{s\in(t_{k}^{n},t_{k+1}^{n}]\}}~\Pi_{\mathcal{B}^{\varepsilon,2}}\Bigg[\frac{n}{T-t}\int_{t_{k-1}^{n}}^{t_k^{n}}\sigma^{\mathbb{P}}_sds\Bigg].
	\end{aligned}
\end{align} 

\begin{rem} 
	Fix any $\varepsilon < \lambda_{\min}(\sigma^o)$ and recall ${\cal P}^\varepsilon(t,x)$ given in \eqref{eq:uncertain.weak}. Then under any $\mathbb{P}\in{\cal P}^\varepsilon(t,x)$, 
	since $\|\sigma^{\mathbb{P}}_s-\sigma^o \|_{\operatorname{F}}< \lambda_{\min}(\sigma^o)$ ${\mathbb{P}}\otimes ds$-a.e., by Remark~\ref{rem:sigma.inverse}, there exists the corresponding inverse matrix $((\sigma^{\mathbb{P}}_s)^{-1})_{s\in[t,T]}$ 
	${\mathbb{P}}\otimes ds$-almost every $(\omega,s)\in\Omega^{0,x}\times [t,T]$. 
	Therefore, if we denote by ${M}^{\mathbb{P},t}:=({M}_s^{\mathbb{P},t})_{s\in [t,T]}$ the $(\mathbb{F}^{X},\mathbb{P})$-local martingale term of $({X}_s)_{s\in [t,T]}$ satisfying $M_t^{\mathbb{P},t}=0$, an application of L\'evy's theorem ensures that for $s\in[t,T]$,
	\begin{align}\label{eq:bm.const}
		{W}_s^{\mathbb{P},t}:= \int_t^s ({\sigma}_u^{\mathbb{P}})^{-1}d{M}_u^{\mathbb{P},t}
	\end{align}
	is a $d$-dimensional  Brownian motion defined on $[t,T]$ under ${\mathbb{P}}$ satisfying ${W}_t^{\mathbb{P},t}=0$. 
\end{rem}

For every $\varepsilon< \lambda_{\min}(\sigma^o)$ and $\mathbb{P}\in{\cal P}^\varepsilon(t,x)$, the piecewise constant processes $b^{\mathbb{P},(n)}$ and $\sigma^{\mathbb{P},(n)}$, $n\in\mathbb{N}$, given in  \eqref{eq:discret3} and the Brownian motion ${W}^{\mathbb{P},t}=(W^{\mathbb{P},t}_s)_{s\in[t,T]}$ given in \eqref{eq:bm.const} enable to define $X^{\mathbb{P},(n)}$ (that is defined on $[0,T]$) for every $n\in\mathbb{N}$ by letting 
\begin{align}\label{eq:X.const}
	{X}^{\mathbb{P},(n)}:=x\oplus_t \left(x+ \int_t^\cdot {b}_s^{\mathbb{P},(n)}ds + \int_t^\cdot {\sigma}_s^{\mathbb{P},(n)}d{W}_s^{\mathbb{P},t}\right).
\end{align}

Finally, for any $\mathbb{P}\in {\cal P}(\Omega^{0,x})$ denote\footnote{See, e.g. \cite[Chapter IV.2, p.124]{protter2005stochastic} for the definition.} by ${\cal H}^2(\Omega^{0,x},{\cal F}^{0,x}, \mathbb{F}^{X}, \mathbb{P})$ the space of all semimartingales $S$ defined on $[0,T]$ such that 
\begin{align}\label{eq:semimtg.space}
	\lVert S \rVert_{{\cal H}^{2}_\mathbb{P}} := \mathbb{E}^{\mathbb{P}}\left[ \langle N,N \rangle_T \right]^{{1}/{2}}+  \mathbb{E}^{\mathbb{P}}\Bigg[ \Big( \int_0^T \lvert dA_t \rvert \Big)^{2} \Bigg]^{{1}/{2}}<\infty,
\end{align}
where $N=(N_t)_{t\in[0,T]}$ and $A=(A_t)_{t\in[0,T]}$ denote the $(\mathbb{F}^X,\mathbb{P})$-local martingale and $\mathbb{F}^{X}$-predictable finite variation process of $S$, respectively (i.e., the canocial decomposition). 

\begin{lem}\label{lem:discret}
	Suppose that Assumption \ref{as:sigma.inverse} is satisfied. Let $\varepsilon < \lambda_{\min}(\sigma^o)$, $(t,x)\in[0,T)\times \mathbb{R}^d$, and $\mathbb{P}\in {\cal P}^\varepsilon(t,x)$. Let $(X^{\mathbb{P},(n)})_{n\in\mathbb{N}}$ be the sequence defined in \eqref{eq:X.const}. Then ${X}^{\mathbb{P},(n)}$ converges to $X$ in ${\cal H}^2({\Omega}^{0,x},{\cal F}^{0,x},\mathbb{F}^{X},{\mathbb{P}})$, i.e. as $n\rightarrow \infty$
	\begin{align*}
		\|{X}^{\mathbb{P},(n)}-X\|_{{\cal H}^{2}_{\mathbb{P}}}\rightarrow 0.
	\end{align*}
\end{lem}
\begin{proof} Let ${b}^{\mathbb{P}}$ be the first differential characteristic of $X$ and ${\sigma}^\mathbb{P}$ be given in \eqref{eq:sqrt_2nd_char}. Using the Brownian motion $W^{\mathbb{P},t}$ defined in~\eqref{eq:bm.const}, we have that $\mathbb{P}$-a.s.
	\begin{align}\label{eq:discret2}
		{X} = x \oplus_t\left(x+\int_t^\cdot {b}_s^{\mathbb{P}} ds+\int_t^\cdot \sigma^{\mathbb{P}}_s d{W}^{\mathbb{P},t}_s\right).
	\end{align}
	That is, the canonical process $X$ can be represented by an It\^o $(\mathbb{F}^{X},\mathbb{P})$-semimartingale with constant $x$-path up to time $t$. 
	
	Recall $\| \cdot \|_{{\cal H}^2_{\mathbb{P}}}$ defined in~\eqref{eq:semimtg.space} and $(b^{\mathbb{P},(n)}, \sigma^{\mathbb{P},(n)})$, $n\in\mathbb{N}$, defined in~\eqref{eq:discret3}. By \eqref{eq:discret2}, H\"older's inequality (with exponent $2$) ensures that for every $n\in \mathbb{N}$,
	\begin{align}\label{eq:estimate.norm}
		\big\lVert {X}^{\mathbb{P},(n)}-X \big\rVert_{{\cal H}^{2}_{\mathbb{P}}}\leq   \mathbb{E}^{{\mathbb{P}}}\Bigg[\int_t^T\| {\sigma}_s^{\mathbb{P},(n)}-{\sigma}_s^{\mathbb{P}} \|_{\operatorname{F}}^2 ds \Bigg]^{\frac{1}{2}}+(T-t)\;\mathbb{E}^{{\mathbb{P}}}\Bigg[\int_t^T\lvert {b}_s^{\mathbb{P},(n)}-{b}^{{\mathbb{P}}}_s \rvert^2 ds\Bigg]^{\frac{1}{2}}.
	\end{align} 
	In particular, by the definition of $\sigma^{\mathbb{P},(n)}$ in \eqref{eq:discret3}, $\int_t^T\|{\sigma}_s^{\mathbb{P},(n)}-{\sigma}_s^{\mathbb{P}}\|_{\operatorname{F}} ds\rightarrow 0$ as $n\rightarrow \infty$ for every $\omega \in \Omega^{0,x}$. Furthermore, since ${\sigma}^{\mathbb{P},(n)}$,  ${\sigma}^{\mathbb{P}}$ are uniformly bounded, the dominated convergence theorem implies that the first term of the right hand side of~\eqref{eq:estimate.norm} vanishes as $n\rightarrow \infty$. The same arguments ensure that the second term vanishes. This completes the proof. 
\end{proof}

\begin{lem}\label{lem:random}
Suppose that Assumption \ref{as:sigma.inverse} is satisfied. Let $\varepsilon < \lambda_{\min}(\sigma^o)$, $(t,x)\in[0,T)\times \mathbb{R}^d$, and $\mathbb{P}\in {\cal P}^\varepsilon(t,x)$. Let $W^{\mathbb{P},t}$ and $({X}^{\mathbb{P},(n)})_{n\in\mathbb{N}}$ be 
given in \eqref{eq:bm.const} and \eqref{eq:X.const}. Then for each $n\in\mathbb{N}$ and $p\geq 1$, the law of ${X}^{\mathbb{P},(n)}$ is contained in the $\tau_p$-closure of the convex hull of the laws of 
\begin{align}\label{eq:random3}
	\left\{x\oplus_t \left(x+\int_t^\cdot {\mu}^u(s,{W}^{\mathbb{P},t})ds+ \int_t^\cdot {\Sigma}^v(s,{W}^{\mathbb{P},t})d{W}^{\mathbb{P},t}_s\right)\Big|(u,v)\in (0,1)^{nd} \times(0,1)^{nd^2} \right\},
\end{align}
where for every $(u,v)\in(0,1)^{nd}\times (0,1)^{nd^2}$, ${\mu}^u:[t,T]\times \Omega^{t}\rightarrow {\cal B}^{\varepsilon,1}$ and ${\Sigma}^{v}:[t,T]\times \Omega^{t}\rightarrow {\cal B}^{\varepsilon,2}$ are adapted\footnote{An adapted functional on $\Omega^t$ is a mapping $\theta:[t,T]\times \Omega^{t}\rightarrow \mathbb{R}$ such that $\theta(s,\cdot)$ is ${\cal F}_s^{W^t}$-measurable for every $s\in[t,T]$ (noting that $\mathbb{F}^{W^t}$ is the raw filtration of the canonical process $W^t$ defined on $[t,T]$; see Section~\ref{sec:thm:main}).~Similarly, an ($\mathbb{R}^d$-valued) adapted functional on $\Omega^t$ is a mapping $\Theta:=(\theta_1,\dots,\theta_d)^\top:[t,T]\times \Omega^{t}\rightarrow \mathbb{R}^d$ such that each $\theta_i$, $i=1,\dots,d$,
	is an adapted functional on~$\Omega^t$.} Borel functionals on $\Omega^t$.
\end{lem}
\begin{proof}
Fix $n\in\mathbb{N}$ and 
denote by $(\widehat{\Omega},\widehat{\cal F},\widehat{\mathbb{F}}:=(\widehat{\cal F}_t)_{t\in[0,T]},\widehat{\mathbb{P}})$ another filtered probability space which carries a $d$-dimensional Brownian motion $\widehat{W}^{t}$ defined on $[t,T]$ satisfying $\widehat{W}^{t}_t=0$ and a sequence $\{(U^{k},V^{k})\;|\;1\leq k\leq n\}$ of $(\mathbb{R}^d,\mathbb{R}^{d\times d})$-valued random variables such that the components $\{(U_i^{k},V_{j,l}^{k})\;|\;1\leq k\leq n; 1\leq i,j,l\leq d\}$ are i.i.d.\;uniformly distributed on $(0,1)$ and independent of $\widehat{W}^{t}$. For notational simplicity, set
\[
U:=(U^1,\dots,U^n),\qquad V:=(V^1,\dots,V^n). 
\]

\vspace{0.5em}
Recall $({b}^{\mathbb{P},(n)},{\sigma}^{\mathbb{P},(n)})$ and $X^{\mathbb{P},(n)}$ given in  \eqref{eq:discret3} and \eqref{eq:X.const}. For each $k=1,\dots,n$, denote by $C([t,t_{k}^n];\mathbb{R}^d)$ the set of all $\mathbb{R}^d$-valued, continuous functions on $[t,t_k^{n}]$ (recalling that $t_k^{n}=t+\frac{k(T-t)}{n}$ with $k=0,1,\dots,n$). Then the existence of regular conditional probability distributions guarantees that there exist measurable functions for every $k=1,\dots,n$ 
\begin{align*}
	\Theta^{1}_k:C([t,t_{k}^{n}];\mathbb{R}^d)\times (0,1)^{k d}\rightarrow \mathcal{B}^{\varepsilon,1},\qquad
	\Theta^{2}_k:C([t,t_{k}^{n}];\mathbb{R}^d)\times (0,1)^{k d^2}\rightarrow \mathcal{B}^{\varepsilon,2},
\end{align*}
such that the random variables defined by
\begin{align}\label{eq:regul.cond}
	\widehat{b}^{(n)}(k):= \Theta^{1}_k\left(\widehat{W}^{t}|_{[t,t_{k}^{n}]},U^{1},\dots,U^{k}\right),\;\; \widehat{\sigma}^{(n)}(k):= \Theta^{2}_k\left(\widehat{W}^{t}|_{[t,t_{k}^{n}]},V^{1},\dots,V^{k}\right)
\end{align}
satisfy
\begin{align}\label{eq:random1}
	\begin{aligned}
		&\mbox{law of $\left\{\widehat{W}^{t},(\widehat{b}^{(n)}(1),\widehat{\sigma}^{(n)}(1)),\dots,(\widehat{b}^{(n)}(n),\widehat{\sigma}^{(n)}(n))\right\}\;$ under $\;\widehat{\mathbb{P}}$}\\
		&=\mbox{law of $ \left\{{W}^{\mathbb{P},t},({b}^{\mathbb{P},(n)}_{t_1^{n}},{\sigma}^{\mathbb{P},(n)}_{t_1^{n}}),\dots,({b}^{\mathbb{P},(n)}_{t_n^{n}},{\sigma}^{\mathbb{P},(n)}_{t_n^{n}})\right\}\;$ under $\;{\mathbb{P}}$}.
	\end{aligned}
\end{align}

\vspace{0.5em}
Now for each $(u,v):=(u^1,\dots,u^n)\times (v^1,\dots,v^n)\in (0,1)^{nd} \times(0,1)^{n d^2}$ and $k=1,\dots,n$, set 
\begin{align}\label{eq:regul.cond.fix}
	\widehat{b}^{(n)}(k;u):= \Theta^{1}_k\left(\widehat{W}^{t}|_{[t,t_k^n]},u^1,\dots,u^k\right),\;\; \widehat{\sigma}^{(n)}(k;v):= \Theta^{2}_k\left(\widehat{W}^{t}|_{[t,t_k^n]},v^1,\dots,v^k \right),
\end{align}
and denote by $\widehat{b}^{(n);u}$ and $\widehat{\sigma}^{(n);v}$ other piecewise constant processes defined on $[t,T]$ such that
\begin{align*}
	\begin{aligned}
		\widehat{b}_s^{(n);u}&:= {\bf 1}_{\{s\in [t,t_1^{n}]\}}~{b}^{o} +\sum_{k=1}^{n-1}{\bf 1}_{\{s\in (t_{k}^n,t_{k+1}^n]\}}\; \widehat{b}^{(n)}(k;u),\\
		\widehat{\sigma}_s^{(n);v}&:= {\bf 1}_{\{s\in [t,t_1^{n}]\}}~{\sigma}^{o} +\sum_{k=1}^n{\bf 1}_{\{s\in (t_{k}^n,t_{k+1}^n]\}}\; \widehat{\sigma}^{(n)}(k;v).
	\end{aligned}
\end{align*}

With the notations in place, we define for each $(u,v)\in (0,1)^{nd} \times(0,1)^{nd^2}$
\begin{align}\label{eq:random2}
	\widehat{X}^{(n);u,v} = x \oplus_t\Bigg(x+ \int_t^\cdot \widehat{b}^{(n);u}_s ds + \int_t^\cdot  \widehat{\sigma}^{(n);v}_s d\widehat{W}_s^{t}\Bigg).
\end{align}

Note that for every $k=1,\dots,n$ (see \eqref{eq:regul.cond} and \eqref{eq:regul.cond.fix}) 
\[
\widehat{b}^{(n)}(k;U)=\widehat{b}^{(n)}(k),\qquad\widehat{\sigma}^{(n)}(k;V)=\widehat{\sigma}^{(n)}(k).
\]
Then by the definitions of $X^{\mathbb{P},(n)}$ and $\widehat{X}^{(n);u,v}$ (see \eqref{eq:X.const} and \eqref{eq:random2}) and the property given in \eqref{eq:random1}, 
\[
\;\mbox{law of $\widehat{X}^{(n);U,V}\;$ under $\;\widehat{\mathbb{P}}$}=\mbox{law of $X^{\mathbb{P},(n)}\;$ under $\;{\mathbb{P}}$}.
\]

Furthermore, as the random variables $(\widehat{b}^{(n)}(k), \widehat{\sigma}^{(n)}(k))$ and $(\widehat{b}^{(n)}(k;u), \widehat{\sigma}^{(n)}(k;v))$ are uniformly bounded for every $k=1,\dots,n$ and $(u,v)\in(0,1)^{nd}\times(0,1)^{nd^2}$, by using the same arguments given in Lemma \ref{lem:priori_weak}, the following holds for every $p\geq 1$,
\begin{align}\label{eq:unif_integ_p}
	\sup_{(u,v)\in (0,1)^{nd}\times(0,1)^{nd^2}}\mathbb{E}^{\widehat{\mathbb{P}}}\left[\sup_{0\leq t\leq T} \left|\widehat{X}_t^{(n);u,v}\right|^p\right]+\mathbb{E}^{\widehat{\mathbb{P}}}\left[\sup_{0\leq t\leq T} \left|\widehat{X}_t^{(n);U,V}\right|^p\right]<\infty.
\end{align}

Therefore, an application of Fubini theorem and \eqref{eq:unif_integ_p} ensure that for every $p\geq 1$ and $\xi\in C_p({\Omega}^{0,x};\mathbb{R})$ (see \eqref{eq:p.mmt.set} for the definition)
\begin{align}\label{eq:HahnBanach1}
	\mathbb{E}^{{\mathbb{P}}}\left[\xi\left({X}^{\mathbb{P},(n)}\right)\right]= \mathbb{E}^{\widehat{\mathbb{P}}}\left[\xi\left(\widehat{X}^{(n);U,V}\right)\right]
	\leq \sup_{(u,v)\in (0,1)^{nd}\times(0,1)^{nd^2}}\mathbb{E}^{\widehat{\mathbb{P}}}\left[\xi\left(\widehat{X}^{(n);u,v}\right)\right]<\infty.
\end{align}


Furthermore, from \eqref{eq:unif_integ_p}, it follows that the laws of $\widehat{X}_t^{(n);U,V}$ and $(\widehat{X}_t^{(n);u,v})$, $(u,v)\in (0,1)^{nd}$ $\times(0,1)^{nd^2}$, belong to ${\cal P}^p(\Omega^{0,x})$ for every $p\geq 1$ (that is equipped with the topology~$\tau_p$; see \eqref{eq:p.msr.set} and \eqref{eq:top.tau.wass}).


{Therefore an application of Hahn-Banach theorem} 
guarantees that for each $n\in\mathbb{N}$ and $p\geq 1$, the law of ${X}^{\mathbb{P},(n)}$ is contained in the $\tau_p$-closure of the convex hull of the laws~of 
\begin{align}\label{eq:random4}
	\left\{x\oplus_t \left(x+\int_t^\cdot {\mu}^u(s,\widehat{W}^{t})ds+ \int_t^\cdot {\Sigma}^v(s,\widehat{W}^{t})d\widehat{W}^{t}_s\right)\;\Big|\;(u,v)\in (0,1)^{nd} \times(0,1)^{nd^2} \right\},
\end{align}
where for every $(u,v)\in(0,1)^{nd}\times (0,1)^{nd^2}$, ${\mu}^u:[t,T]\times \Omega^{t}\rightarrow {\cal B}^{\varepsilon,1}$ and ${\Sigma}^{v}:[t,T]\times \Omega^{t}\rightarrow {\cal B}^{\varepsilon,2}$ are adapted Borel functionals on $\Omega^t$.

Replacing $\widehat{W}^t$ with $W^{\mathbb{P},t}$ in the set \eqref{eq:random4} ensures the claim to hold.
\end{proof}

Define by ${\cal G}([t,T]\times \Omega^{t};\mathcal{B}^{\varepsilon,1})$ the set of all adapted, $\mathcal{B}^{\varepsilon,1}$-valued, Borel functionals on $\Omega^{t}$. Define ${\cal G}([t,T]\times \Omega^{t};\mathcal{B}^{\varepsilon,2})$ analogously, with $\mathcal{B}^{\varepsilon,1}$ replaced by $\mathcal{B}^{\varepsilon,2}$, and set
\begin{align}\label{eq:uncertain.ftn}
{\cal G}^\varepsilon(t):= \left\{(\mu,\Sigma)\in {\cal G}([t,T]\times \Omega^{t};\mathcal{B}^{\varepsilon,1})\times {\cal G}([t,T]\times \Omega^{t};\mathcal{B}^{\varepsilon,2}) \right\}.
\end{align}

\begin{proof}[Proof of Proposition \ref{pro:weak.dense}]
Recall the Wiener measure $\mathbb{P}_0^t$ defined on $(\Omega^t,{\cal F}^{t},\mathbb{F}^{W^t})$ under which the canonical process $W^t=(W^t_s)_{s\in[t,T]}$ is a Brownian motion satisfying $W^t_t=0$ (see Section~\ref{sec:thm:main}). Moreover, recalling the set ${\cal G}^\varepsilon(t)$ given in \eqref{eq:uncertain.ftn}, we denote~by 
\begin{align}\label{eq:char.strong.subset}
	{\cal D}^{\varepsilon}(t):=\left\{(b,\sigma)\;\Big|\;\mbox{$(b_s,\sigma_s):=(\mu(s,W^t),\Sigma(s,W^t))$\;\;\mbox{for $s\in[t,T]$},\;\;$(\mu,\Sigma)\in {\cal G}^\varepsilon(t)$}\right\},
\end{align}
and we define 
\begin{align}\label{eq:uncertain.strong}
	{\cal Q}^\varepsilon_{\operatorname{sub}}(t,x):={\cal Q} \left({t,x};{\cal D}^\varepsilon \right):=\left\{\mathbb{P}_0^t\circ \left(x\oplus_t {X}^{t,x;b,\sigma}  \right)^{-1}\;\Big|
	\;\mbox{$(b,\sigma)\in {\cal D}^\varepsilon(t)$}\right\},
\end{align}
where ${X}^{t,x;b,\sigma}$ is defined in \eqref{eq:ito.semi.time}.

From the definition of ${\cal G}^\varepsilon(t)$, it follows that ${\cal D}^{\varepsilon}(t)\subseteq {\cal C}^{\varepsilon}(t)$ (see \eqref{eq:char.strong}). Furthermore, since ${\cal Q}^\varepsilon(t,x)={\cal Q} \left({t,x};{\cal C}^{\varepsilon} \right)$ (see \eqref{eq:uncertain.strong1}), by Remark \ref{rem:inclusion} we have
\[
{\cal Q}^\varepsilon_{\operatorname{sub}}(t,x)\subseteq {\cal Q}^\varepsilon(t,x)\subseteq {\cal P}^\varepsilon(t,x).
\]

Now let $\mathbb{P}\in {\cal P}^\varepsilon(t,x)$. 
Lemma \ref{lem:discret} ensures that ${X}^{\mathbb{P},(n)}$, $n\in\mathbb{N}$, given in \eqref{eq:X.const}, converges to the canonical process $X$ in ${\cal H}^2({\Omega^{0,x}},{\cal F}^{0,x},\mathbb{F}^{X},\mathbb{P})$. 
Thus Lemma~\ref{lem:random} together with \eqref{eq:char.strong.subset} and \eqref{eq:uncertain.strong} ensures that $\mathbb{P}$ is contained in the $\tau_p$-closure of the convex hull of ${\cal Q}^\varepsilon_{\operatorname{sub}}(t,x)$ for every~$p\geq 1$.  
\end{proof}

\newpage 

\bibliographystyle{abbrv}
\bibliography{references}
     
\end{document}